\documentclass[12pt]{amsart}

\usepackage{amssymb,upref,enumerate}
\usepackage{amsmath,amssymb,amscd,amsthm,amsxtra}
\usepackage{latexsym}
\usepackage{times}

\calclayout

\usepackage[colorlinks=true, pdfstartview=FitV, linkcolor=blue,
citecolor=blue, urlcolor=blue]{hyperref}

\def\Xint#1{\mathchoice
{\XXint\displaystyle\textstyle{#1}}%
{\XXint\textstyle\scriptstyle{#1}}%
{\XXint\scriptstyle\scriptscriptstyle{#1}}%
{\XXint\scriptscriptstyle\scriptscriptstyle{#1}}%
\!\int}
\def\XXint#1#2#3{{\setbox0=\hbox{$#1{#2#3}{\int}$}
\vcenter{\hbox{$#2#3$}}\kern-.5\wd0}}

\def\dashint{\Xint-}

\newcommand{\ra}{\rightarrow}

\newcommand{\bey}{\begin{eqnarray*}}
\newcommand{\eey}{\end{eqnarray*}}
\newcommand{\beal}{\begin{align*}}
\newcommand{\enal}{\end{align*}}
\newcommand{\be}{\begin{equation}}
\newcommand{\ee}{\end{equation}}
\newcommand{\R}{\mathbb R}

\newcommand{\N}{\mathbb N}

\newcommand{\Lap}{\mathcal L }

\newcommand{\ep}{\epsilon}

\newcommand{\bc}{\begin{center}}
\newcommand{\ec}{\end{center}}

\newcommand{\vp}{\varphi}

\newcommand{\gensub}{U} 

\DeclareMathOperator{\supp}{supp}
\DeclareMathOperator{\Div}{div}
\DeclareMathOperator*{\esssup}{ess\,sup}
\DeclareMathOperator*{\essinf}{ess\,inf}
\DeclareMathOperator{\lip}{Lip}
\DeclareMathOperator{\osc}{osc}
\newcommand{\f}{\mathbf{f}}
\DeclareMathOperator{\dist}{dist}

\newcommand{\grad}{\nabla}

\newtheorem{theorem}{Theorem}[section]
\newtheorem{nonum}{Theorem}

\newtheorem{lemma}[theorem]{Lemma}
\newtheorem{corollary}[theorem]{Corollary}
\newtheorem{prop}[theorem]{Proposition}
\newtheorem*{nonum-prop}{Proposition}

\allowdisplaybreaks

\theoremstyle{definition}

\newtheorem{example}[theorem]{Example}

\newtheorem{definition}[theorem]{Definition}
\theoremstyle{remark}
\newtheorem{remark}[theorem]{Remark}

\numberwithin{equation}{section}

\begin{document}


\subjclass[2000]{Primary 42A50}

\keywords{$p$-Laplacian, H\"ormander vector fields, $A_p$ weights}

\thanks{The first author is supported by the Stewart-Dorwart Faculty Development Fund
 at Trinity College and
 grant MTM2012-30748 from the Spanish Ministry of Science and
 Innovation.  The second and third authors are partially supported by the NSF under grants DMS 1201504 and DMS 1101327 respectively}

\author{David Cruz-Uribe, SFO}
\address{Department of Mathematics\\
Trinity College, Hartford, CT 06106, USA.}\email{david.cruzuribe@trincoll.edu}

\author{Kabe Moen}
\address{Department of Mathematics\\
University of Alabama, Box 870350, 345 Gordon Palmer Hall.}\email{kmoen@as.ua.edu}

\author{Virginia Naibo}
\address{Department of Mathematics\\
Kansas State University, 138 Cardwell Hall, Manhattan, KS 66506
USA.}\email{vnaibo@math.ksu.edu}

\title[Regularity of solutions to $p$-Laplacian equations]{Regularity of solutions to degenerate $p$-Laplacian equations}

\date{September 26, 2012}

\begin{abstract}
We prove regularity results for solutions of the equation
\[  \Div\big(\langle AXu, X u\rangle^{\frac{p-2}{2}}
AX u\big) = 0, \]
$1<p<\infty$, where $X=(X_1,\ldots,X_m)$ is a family of vector fields
satisfying H\"ormander's condition, and $A$ is an $m\times m$ symmetric matrix that satisfies degenerate
ellipticity conditions.  If the degeneracy is of the form
\[ \lambda w(x)^{2/p}|\xi|^2\leq \langle A(x)\xi,\xi\rangle\leq
\Lambda w(x)^{2/p}|\xi|^2,
\]
$w \in A_p$, then we show that solutions are locally H\"older
continuous.  If the degeneracy is of the form 
\[ k(x)^{\frac{-2}{p'}}|\xi|^2\leq \langle A(x)\xi,\xi\rangle\leq
k(x)^{\frac2p}|\xi|^2,
\]
$k\in A_{p'}\cap RH_\tau$,where $\tau$ depends on the homogeneous
dimension,  then the solutions are continuous almost
everywhere, and we give examples to show that this is the best
result possible.  We give an application to maps of finite distortion.
\end{abstract}

\maketitle


\section{Introduction and main results}

In this paper we consider the regularity of solutions of degenerate quasilinear elliptic
equations that generalize the $p$-Laplacian
\begin{equation} \label{eqn:p-laplacian}
 \bigtriangleup_p u = \Div(|\grad u|^{p-2} \grad u) = 0,
\end{equation}
and the degenerate linear elliptic equations of the form
\begin{equation} \label{eqn:degen-elliptic}
 L u = \Div(A\grad u) = 0,
\end{equation}
where $A$ is an $n\times n$ real symmetric matrix that satisfies the
ellipticity condition
\begin{equation} \label{eqn:wtd-elliptic}
 \lambda w(x)|\xi|^2 \leq \langle A(x)\xi,\xi \rangle \leq \Lambda w(x)|\xi|^2,
\end{equation}
with $w$  in the Muckenhoupt class $A_2$.  In the simplest
case we prove the following results.  Throughout this paper, 
$\Omega$ will denote either a bounded open, connected subset of $\R^n$ or all
of $\R^n$.  For precise
definitions of the weight classes $A_{p}$ and $RH_t$, see
Section~\ref{tools} below.  

\begin{nonum}
Given $\Omega$ and $p$, $1<p<\infty$, 
let $A$ be an $n\times n $ real-symmetric matrix of measurable
functions defined in $\Omega$ such that for every $\xi\in \R^n$ and a.e. $x\in \Omega $,
\begin{equation} \label{eqn:ellipticityA}
\lambda w(x)^{2/p}|\xi|^2\leq \langle A(x)\xi,\xi\rangle\leq
\Lambda w(x)^{2/p}|\xi|^2,
\end{equation}
where $w\in A_p$ and $0<\lambda <\Lambda<\infty$.   Then every 
weak solution $u$ of the equation
\begin{equation} \label{eqn:diffA}
 \Lap_p u(x)  = -\Div\big(\langle A(x)\grad u(x), \grad u(x)\rangle^{\frac{p-2}{2}}
A(x)\grad u(x)\big) = 0
\end{equation}
is H\"older continuous on compact subsets of $\Omega$.   
\end{nonum}

\begin{nonum}
Given $\Omega$ and $p$, $1<p<\infty$, let $A$ be an $n\times n $ real-symmetric matrix of measurable
functions defined in $\Omega$ such that for every $\xi\in \R^n$ and a.e. $x\in \Omega$, 
\begin{equation} \label{eqn:ellipticityB}
k(x)^{\frac{-2}{p'}}|\xi|^2\leq \langle A(x)\xi,\xi\rangle\leq
k(x)^{\frac2p}|\xi|^2,
\end{equation}
where $k\in A_{p'}\cap RH_n$.  Then every weak solution $u$  of \eqref{eqn:diffA}
is continuous almost everywhere in $\Omega$.
\end{nonum}

Continuity almost everywhere is the best result possible with these
hypotheses.  To illustrate this, consider the following example
adapted from \cite{MR847996}.  Let $p=2$, let $\Omega$ be the unit ball in $\R^2$
and fix
$q>2$.  Define $k(x,y)=|x|^{-1/q}$.  Then $k\in A_1\cap RH_2$.  
Define the matrix $A$ by

\[ A(x) = \left(\begin{matrix} k(x)^{-1} & 0 \\ 0 &
    k(x) \end{matrix}\right);  \]
then $A$ satisfies the ellipticity condition \eqref{eqn:ellip:p=2}.   The function
\[ u(x,y) = \frac{x}{|x|}\exp(|x|^{1/q'}) \sin(y/q') \]
is a (formal) solution of \eqref{eqn:diffA}  but it is discontinuous
on the $y$-axis.   As we will see below in Section~\ref{proofofmain}, our results actually give us
precise control on the set of discontinuities of $u$.  

To put these theorems in context, we describe some known results.  It
is well-known that solutions to the $p$-Laplacian
\eqref{eqn:p-laplacian} are in $C_{loc}^{1,\alpha}$ for some $\alpha$,
$0<\alpha<1$ (see~\cite{MR672713,MR721568,MR0244628}).  Fabes, Kenig
and Serapioni~\cite{MR643158} showed that given
\eqref{eqn:wtd-elliptic}, the solutions of ~\eqref{eqn:degen-elliptic}
are H\"older continuous.  Modica~\cite{MR839035} proved the Harnack
inequality and H\"older continuity for quasi-minimizers of certain
variational integrals controlled by Muckenhoupt weights and obtained a
version of Theorem~A.

In the more general case when the matrix $A$ satisfies the two-weight
ellipticity condition 
\[  w(x)|\xi|^2 \leq \langle A(x)\xi,\xi \rangle \leq v(x)|\xi|^2, \]
where $w(x)\leq v(x)$ are non-negative measurable functions, then in
general solutions of ~\eqref{eqn:degen-elliptic} will not be regular.  However, Chanillo and
Wheeden~\cite{MR847996} proved that if $v,\,w$ are
doubling, $w\in A_2$, and there exists $q>2$ such that given two balls
$B(x_1,r_1)\subset B(x_2,r_2)$,
\begin{equation} \label{eqn:cw-balance}
\frac{r_1}{r_2}\left(\frac{v(B(x_1,r_1))}{v(B(x_2,r_2))}\right)^{1/q}
\leq C \left(\frac{w(B(x_1,r_1))}{w(B(x_2,r_2))}\right)^{1/2},
\end{equation}
then solutions of ~\eqref{eqn:degen-elliptic} satisfy a
Harnack inequality.  As a consequence, they showed that if we define
\[ \mu(x,r) = \left(\frac{v(B(x,r))}{w(B(x,r))}\right)^{1/2}, \]
then the solution is continuous provided that 
\[ \mu(x,r) = o(\log\log(r^{-1})) \]
as $r\rightarrow 0$.  Ferrari~\cite{MR2228656} extended these results
to non-negative solutions of  the degenerate
$p$-Laplacian \eqref{eqn:diffA} when the matrix $A$ satisfies
\[  w(x)^{2/p}|\xi|^2 \leq \langle A(x)\xi,\xi \rangle \leq v(x)^{2/p}|\xi|^2 \]
assuming $v,\,w$ are doubling, $w\in A_p$, $q>p$ and
\eqref{eqn:cw-balance} holds with the exponent $1/2$ replaced by
$1/p$.  

Using the results of Chanillo and Wheeden, Theorem B in the case $p=2$
was recently proved by the first author, di Gironimo and
Sbordone~\cite{cruz-diGironimo-sbordone-P}.      Degenerate elliptic
equations of this type arise in the study of maps of finite
distortion in the plane:  see, for
example,~\cite{MR2472875}.   
%
%
%
%

We will actually prove our results in the more general setting of
$C^\infty$ vector fields.  This introduces a different kind of
degeneracy:  while  the matrix $A$ can be thought of as making the
measure degenerate, the vector fields introduces degeneracies into the geometry.  Let $X_1,\ldots,X_m$ be a family of
$C^\infty$ vector fields on $\R^n$ that satisfy H\"ormander's condition.    Define the operators
$Xu=(X_1u,\ldots,X_mu)$ and 
\[ \Div_X \mathbf{v}= -\sum_{j=1}^m X_i^*v_i, \]
where $\mathbf{v}=(v_1,\ldots,v_m)$.  Then the operators above have
natural generalizations to Carnot-Carath\'eodory spaces, and in
particular to Carnot groups. Operators in this generality have been considered by a number of authors. Capogna, Danielli and Garofalo~\cite{MR1239930} showed that
solutions of 
\[ \Div_X(|Xu|^{p-2}Xu) = 0\]
are H\"older continuous.  Lu~\cite{MR1202416} proved the regularity of
solutions of
\[ \Div_X(A Xu)= 0 \]
where $A$ is an $m\times m$ real-symmetric matrix that satisfies
\eqref{eqn:wtd-elliptic} for $w\in A_2$.  Franchi, Lu and
Wheeden~\cite{MR1343563,MR1354890} extended the results of Chanillo
and Wheeden~\cite{MR847996} to this more general setting.
Ferrari~\cite{MR2228656} in turn extended his results for the
$p$-Laplacian discussed above to the solutions of the equation
\begin{equation}\label{finitedisteqn}
\Div_X(\langle AXu,Xu\rangle^{\frac{p-2}{2}} AXu) = 0. 
\end{equation}

The main results of this paper extend Theorems A and B to the setting
of Carnot-Carath\'eodory spaces.    Theorem~\ref{thm:easy-thm}
generalizes Theorem A, and Theorems~\ref{thm:main-CCS}
and~\ref{thm:main-carnot} generalize Theorem B.  Since the statement
of all of these results requires some additional definitions, we defer
their precise statement until below.   In the end, our main results follow from standard regularity arguments.  The bulk of
our work is to construct the machinery necessary to show that these standard
arguments are actually applicable in this very general setting.  In doing so we draw
upon earlier work by the first author, di Gironimo and Sbordone \cite{cruz-diGironimo-sbordone-P} and Ferrari \cite{MR2228656}.

Our results are applicable to mappings of finite distortion.  A
function $f:\Omega\ra \R^n$ with $f\in W_{loc}^{1,1}(\Omega,\R^n)$ and
locally integrable Jacobian is a mapping of finite distortion if there
is a function $K\geq 1$ finite a.e. such that
\[ |Df(x)|^n\leq K(x)J_f(x), \]
where $|Df(x)|=\sup_{|h|=1} |Df(x)h|$ is the operator norm and
$J_f(x)=\det Df(x).$ These mappings have been extensively studied: see
Astala, Iwaniec, and Martin \cite[Chapter 20]{MR2472875} for mappings
in the plane and Iwaniec and Martin \cite[Chapters 6 and 7]{MR1859913}
for results in higher dimensions.   The regularity of such mappings
has also been extensively addressed.  Mappings of finite distortion satisfy
an equation of the form~\eqref{eqn:diffA} with the degeneracy
controlled by the distortion function.  As an application of our
results we show that with assumptions on the distortion, mappings of
finite distortion are continuous almost everywhere.  This
complements work of Manfredi~\cite{MR1294334}.  As further motivation for the degenerate equations  considered in this article we note that the theory of mappings of finite distortion has been extended to Carnot groups and such mappings satisfy equations of the form \eqref{finitedisteqn}: see for instance Vodop'yanov \cite[Section 4]{MR1721674}.

The remainder of this paper is organized as follows.  In
Section~\ref{tools} we gather some preliminary results.  In \S
\ref{subsec:hormander} we recall the definition and basic properties
of H\"ormander vector fields, including Carnot groups as a particular
case; in \S \ref{subsec:weights} we discuss the theory of Muckenhoupt
weights in the context of spaces of homogeneous type.  In
Section~\ref{section:PDE} we give the machinery we need to work with
the equations we are interested in.  In \S \ref{subsec:soln}- \ref{sec:weaksolutions} we define
precisely the solution to $\Lap_p u=0$.  We follow the definition
given in~\cite{MR2228656}; for the convenience of the reader we
provide full details. In \S\ref{subsec:harnack} we give the Harnack
inequalities we use in our main proofs.  In Section~\ref{proofofmain}
we prove Theorems~\ref{thm:easy-thm}, \ref{thm:main-CCS}
and~\ref{thm:main-carnot}.  In Section~\ref{sharp} we show that
Theorem B is sharp: we construct an operator $\Lap_p$ and a solution
$u$ of $\Lap_p u=0$ that is discontinuous on a set of measure zero.
Finally, in Section~\ref{section:finite-distortion} we give our
applications to the theory of mappings of finite distortion.

Throughout this paper, $C$, $c$, etc. will denote a constant whose values may change from line to line.  Generally, it will depend on the value of $p$, the dimension of the vector field and the characteristics of the weight, but not on the specific function.

{\bf Acknowledgement.}  The authors completed this collaboration while at the  American Institute of Mathematics. They would like to thank AIM for its hospitality and support. They also  thank the anonymous 
referees for  their  detailed comments and suggestions.

\section{Preliminary Results}
\label{tools}

\subsection{H\"ormander vector fields}
\label{subsec:hormander}

Let 
$X=(X_1,\ldots, X_m)$ be a family of  infinitely differentiable vector fields defined in $\R^n$ and with values in $\R^m.$
We identify each $X_j$ with the first order differential operator acting on Lipschitz functions  given  by
\[
X_j f (x)=X_j(x)\cdot \nabla f(x),\quad\,\, j=1,\cdots,m.  
\]
We define $Xf=(X_1f,X_2f,\cdots,X_mf)$ and set
\[
|Xf(x)|=\left(\sum_{j=1}^m|X_jf(x)|^2\right)^{\frac{1}{2}}.
\]

Given an open, connected set $\Omega$, $X$ is said to satisfy H\"ormander's
condition in $\Omega$ if there exists a neighborhood $\Omega_0$ of
$\overline{\Omega}$ and $l\in\N$ such that the family of commutators
of the vector fields in $X$ up to length $l$ span $\R^n$ at every
point of $\Omega_0$.   Hereafter, we will assume that $X$ satisfies
H\"ormander's condition on every bounded, connected subset of $\R^n$.  

Let $C_X$ be the family of absolutely
continuous curves $\zeta:[a,b]\to\R^n,$ $a\le b,$ such that there
exist measurable functions $c_j(t),$ $a\le t\le b,$ $j=1,\cdots, m,$
satisfying $\sum_{j=1}^{m} c_j(t)^2\le 1$ and
$\zeta'(t)=\sum_{j=1}^{m}c_j(t) X_j(\zeta(t))$ for almost every $t\in
[a,b].$ Given $x,\,y\in\Omega$, define
\[\rho(x,y)=\inf\{T>0: \exists \zeta \in C_X \text{ s.t. } \zeta(0)=x, \zeta(T)=y\}.\]
The function $\rho$ is a metric on $\Omega$ called the
Carnot-Carath\'eodory metric associated to $X$, and the pair $(\Omega,\rho)$ is said to be a Carnot-Carath\'eodory space.   We refer the reader to~\cite{MR793239} for more details on
Carnot-Carath\'eodory spaces.


By a ball $B$ of radius $r$ and center $x$ we mean the metric ball $B
= B(x,r) = \{ y \in \R^n : \rho(x,y) < r \}$.  Given a ball $B$, let
$r(B)$ denote its radius. Nagel, Stein and Wainger proved in
\cite{MR793239} that for every compact set $K\subset\Omega$ there
exist positive constants $R_0,$ $c,$ $c_1$ and $c_2$ depending on $K$
such that
\begin{equation}\label{ineq:doubling}
|B(x,2r)|\le c\,|B(x,r)|, \qquad x\in K, \,0<r< R_0
\end{equation}
and
\begin{equation}\label{ineq:rhotopo}
c_1|x-y|\le \rho(x,y)\le c_2|x-y|^{\frac{1}{l}}, \qquad x,\,y\in K.
\end{equation}
Note that inequality \eqref{ineq:rhotopo} implies that the topology
induced by $\rho$ on $\Omega$ is the same as the Euclidean topology.
Using \eqref{ineq:doubling}, Bramanti and Brandolini proved in
\cite[p. 533]{MR2174915} that the triple $(B,\rho,dx)$ is a space of
homogeneous type for all metric balls $B\subset \Omega$; moreover,
structural constants are uniform for all metric balls with radius
bounded by a fixed constant and with center in a fixed compact subset
of $\Omega$.  In addition, they gave regularity properties for
$\partial \gensub,$ $\gensub\subset \Omega$, that imply that $(\gensub,
\rho, dx)$ is a space of homogeneous type. One such condition,
possessed by every metric ball \cite[p. 533]{MR2174915}, is
\begin{equation}\label{eqn:regularitycondition}
|B(x, r)\cap \gensub|\ge C|B(x,r)|, \qquad x\in \gensub, \, 0<r<\text{diam}_\rho(\gensub).
\end{equation}

Of importance below is  the homogeneous dimension of a
Carnot-Carath\'eodory space.  
In the particular case when  $X$ is the generator of
a homogeneous Carnot group on $\R^n$, there exists a constant $Q\geq
n$ such that $|B(x,r)|=|B(0,1)|\,r^Q$;  we will define $Q$ to be the
homogeneous dimension of $(\R^n,\rho)$.  As a consequence, we have
that 
$(\R^n,\rho, dx)$ is a space of homogeneous type.  For a detailed description of 
the theory of Carnot groups see, for instance, the book by Bonfiglioli, Lanconelli 
and Uguzzoni~\cite{MR2363343}. 

More generally, if $\Omega$ is a bounded connected subset of $\R^n$,
then it was shown in~\cite[Proposition 2.1]{MR1239930} (see also~\cite{MR1343563}) that
there exist positive constants $Q$, $R$ and $c$, all depending on
$\Omega$,  such that for all $x\in \Omega$
and $0<s<r<R$
\begin{equation} \label{eqn:lower-bound}
 c\left(\frac{s}{r}\right)^Q \leq \frac{|B(x,s)|}{|B(x,r)|}.
\end{equation}
Hereafter, for such $\Omega$, we define this constant $Q$ to be the
homogeneous dimension.  Unless we are in the special case of a Carnot
group, this value depends on $\Omega$: it is clear from~\cite{MR1239930} that if
$\Omega'\subset \Omega$, then the homogeneous dimension of $\Omega'$
will be smaller than that of $\Omega$.   Moreover, there is a certain ambiguity in this
definition, since given one value of $Q$ we can take any larger
value.   As we will see in Section~\ref{proofofmain} the smallest
possible value should be taken.

The following upper and lower estimates play a significant role in our proofs of
Theorems~\ref{thm:main-CCS} and~\ref{thm:main-carnot}.

\begin{lemma} \label{lemma:hormander-nocenter}
Given any bounded set $\Omega \subset \R^n$ (or $\Omega=\R^n$ if $X$
is the generator of a Carnot group) let $Q$ be the homogeneous
dimension.    Then  there exist positive
constants $d_\Omega,\,D_\Omega,\,R_\Omega$ such that given $x_1,\,x_2\in \Omega$ and
$0<r_1<r_2<R_\Omega$, if $B(x_1,r_1)\subset B(x_2,r_2)$, then
\begin{equation} \label{eqn:hormander-nocenter1}
d_\Omega \left(\frac{r_1}{r_2}\right)^Q \leq \frac{|B(x_1,r_1)|}{|B(x_2,r_2)|} \leq
D_\Omega \left(\frac{r_1}{r_2}\right)^n.
\end{equation}
\end{lemma}

\begin{proof}
Fix $\Omega$; when the balls are concentric, there exist constants $c_\Omega, C_\Omega$ and $r_\Omega$ such that
\begin{equation} \label{eqn:hormander-doubling1}
 c_\Omega \left(\frac{s}{r}\right)^Q \leq \frac{|B(x,s)|}{|B(x,r)|} \leq
C_\Omega \left(\frac{s}{r}\right)^n
\end{equation}
whenever $x\in \Omega$ and $0<s<r<r_\Omega$.  The first inequality is
just~\eqref{eqn:lower-bound} above; the proof of the second is in~\cite[p. 584]{MR1343563}.  Let $R_\Omega=r_\Omega/4$.   Fix balls
$B(x_1,r_1)$ and $B(x_2,r_2)$ as in the hypotheses.  Then
$B(x_2,r_2)\subset B(x_1,2r_2)$, and so by the first inequality
in~\eqref{eqn:hormander-doubling1},
\[ \frac{|B(x_1,r_1)|}{|B(x_2,r_2)|}  \geq
\frac{|B(x_1,r_1)|}{|B(x_1,2r_2)|} \geq
c_\Omega\left(\frac{r_1}{2r_2}\right)^Q, \]
so we get the first inequality in \eqref{eqn:hormander-nocenter1} with
$d_\Omega=2^{-Q}c_\Omega$.

Similarly, we have that $B(x_2,r_2)\subset B(x_1,2r_2)\subset
B(x_2,4r_2)$, and so again by the first inequality in
\eqref{eqn:hormander-doubling1},
\[ 4^Qc_\Omega^{-1} |B(x_2,r_2)|\geq |B(x_2,4r_2)| \geq |B(x_1,2r_2)|.  \]
Hence by the second  inequality in
\eqref{eqn:hormander-doubling1},
\[ \frac{|B(x_1,r_1)|}{|B(x_2,r_2)|}  \leq
 4^Qc_\Omega^{-1} \frac{|B(x_1,r_1)|}{|B(x_1,2r_2)|}  \leq
 4^Qc_\Omega^{-1}C_\Omega\left(\frac{r_1}{2r_2}\right)^n. \]
Therefore, we get the second inequality in
\eqref{eqn:hormander-nocenter1} with
$D_\Omega=2^{2Q-n}c_\Omega^{-1}C_\Omega$.
\end{proof}

\subsection{Muckenhoupt weights in spaces of homogenous type}
\label{subsec:weights}

The theory of
Muckenhoupt weights in spaces of homogenous type is well-developed and
parallels the standard theory.  

By a weight we will mean a positive function in $L^1_{loc}(\R^n)$.  We
will say that a weight $w$ is doubling in $\gensub\subset \R^n$ if
there exists a constant $C_\gensub$ such that
\[ w(B(x,2r)\cap\gensub) \leq C_\gensub\,w(B(x,r)\cap\gensub),\quad x\in \gensub,\,r>0. \]
We will say that $w$ is locally doubling in $U$ if  for each compact set
$K\subset \gensub$ and $R>0$ there exists  $C_{K,R}$ such
that 
\[ w(B(x,2r)\cap\gensub) \leq C_{K,R}\,w(B(x,r)\cap\gensub),\quad x\in
K,\,0<r<R. \]

Given $p$,  $1<p<\infty$, and a weight $w$, we say $w\in A_p(\gensub,\rho,dx)$ if 
\[ [w]_{A_p} = \sup_{\substack{r>0\\x\in \gensub}}
\left(\dashint_{B(x,r)\cap\gensub} w(x)\,dx\right) \left(\dashint_{B(x,r)\cap\gensub}
  w(x)^{1-p'}\,dx\right)^{p-1} < \infty. \]
When $p=1$, we say $w\in A_1(\gensub,\rho,dx)$ if 
\[ [w]_{A_1} = \esssup_{x\in \gensub} \frac{Mw(x)}{w(x)} < \infty, \]
where $M$ is the maximal operator defined using metric balls:  given a
measurable function $f$, for all $x\in \gensub$, 
\[ Mf(x) = \sup_{r>0} \dashint_{B(x,r)\cap\gensub} |f(y)|\,dy. \]
Finally, we set $A_\infty(\gensub,\rho,dx):=\cup_{p=1}^\infty A_p(\gensub,\rho,dx).$

We say that a weight $w$ satisfies a reverse H\"older inequality in $\gensub$ with
exponent $t>1$, denoted by $w\in RH_t(\gensub,\rho,dx)$, if
\[ [w]_{RH_t} = \sup_{\substack{r>0\\x\in \gensub}}
\frac{\left(\dashint_{B(x,r)\cap\gensub} w(x)^t\,dx\right)^{1/t}}{\dashint_{B(x,r)\cap\gensub}
  w(x)\,dx} < \infty.  \]

\begin{remark} \label{rem:restriction}
If $\gensub$ is a bounded set that satisfies
\eqref{eqn:regularitycondition} then $|B|\sim  |B\cap \gensub|$ for any
metric ball with center in $\gensub$ and $r(B)\le \text{diam}_\rho(\gensub).$ Therefore, if $w\in
A_p(\R^n,\rho,dx)$, then the restriction of $w$ to $\gensub$ belongs to
$A_p(\gensub,\rho,dx).$ In particular, this is true when $\gensub$ is a metric
ball.
\end{remark}

Muckenhoupt weights have the following properties.  

\begin{lemma} \label{prop-Apfacts} 
  Let $\gensub\subset \R^n$ and suppose that $(\gensub,\rho,dx)$ is a space of homogeneous type.  If $w\in A_p(\gensub,\rho,dx)$ for
  some $p>1$, then the following are true:

\begin{enumerate}[(a)]
\item \label{prop-Apfacts-a} $w^{1-p'} \in A_{p'}(\gensub,\rho,dx)$;

\item\label{prop-Apfacts-b} there exists $t>1$ such that $w\in
  RH_t(\gensub,\rho,dx)$;

\item\label{prop-Apfacts-c} there exists $\epsilon>0$ such that $w\in A_{p-\epsilon}(\gensub,\rho,dx)$.
\end{enumerate}

\end{lemma}

As in the Euclidean case, part \eqref{prop-Apfacts-a} is a direct
consequence of the definition of $A_p$ weight, part (b) is proved in~\cite{MR1011673}, and part
\eqref{prop-Apfacts-c} easily follows from \eqref{prop-Apfacts-a} and
\eqref{prop-Apfacts-b}.

\begin{lemma} \label{lemma-improve} Let $\gensub\subset \R^n.$ 
\begin{enumerate}[(a)]
\item\label{lemma-improve-a} If $w\in RH_t(\gensub,\rho,dx)$ for some $t>1$, then
for any $x\in \gensub,$ $r> 0$, and any measurable set $E\subset B(x,r)\cap \gensub,$
\[ \frac{w(E)}{w(B(x,r)\cap \gensub)} \leq [w]_{RH_t(\gensub,\rho,dx)}
\left(\frac{|E|}{|B(x,r)\cap \gensub|}\right)^{1/t'}. \]
\item\label{lemma-improve-b}  If $w\in A_p(\gensub,\rho,dx)$ for
  some $p>1$, then
for any $x\in \gensub,$ $r> 0$, and any measurable set $E\subset B(x,r)\cap \gensub,$
\[ \frac{|E|}{|B(x,r)\cap \gensub|}\leq  \left([w]_{A_p(\gensub,\rho,dx)}
\frac{w(E)}{w(B(x,r)\cap \gensub)} \right)^{1/p}. \]
In particular, if  $(\gensub,\rho,dx)$ is a space of homogeneous type, then $w$ is doubling:
   for all $x\in \gensub$ and $r>0$,
\[ w(B(x,2r)\cap \gensub) \leq D^p\,[w]_{A_p(\gensub,\rho,dx)}w(B(x,r)\cap \gensub),\]
where $D$ is a doubling constant for $(\gensub,\rho,dx).$ 
\end{enumerate}
\end{lemma}

In the classical case these estimates can be found in
\cite{garcia-cuerva-rubiodefrancia85}.  
The proofs in the setting of spaces of
homogeneous type are essentially the same, replacing Euclidean
balls/cubes with metric balls.  The doubling property is a consequence
of the first inequality in part~\eqref{lemma-improve-b} and the fact that
$(\gensub,\rho,dx)$ is a space of homogenous type.

\begin{remark}\label{remark:constants}  
  In Lemma~\ref{prop-Apfacts}, we note that
  $[w^{1-p'}]_{A_{p'}(\gensub,\rho,dx)}=[w]_{A_p(\gensub,\rho,dx)}^{{1}/{(p-1)}}$
  and the constants $t,$ $[w]_{RH_t(\gensub,\rho,dx)},$
  $\epsilon,$ and $[w]_{A_{p-\epsilon}(\gensub,\rho,dx)}$ depend only
  on $p$,  bounds for $[w]_{A_p(\gensub,\rho,dx)}$ and structural
  constants of $(\gensub,\rho,dx).$ In particular, if $w\in
  A_p(\Omega,\rho,dx),$ all these constants are uniformly bounded on
  $\gensub$ when $\gensub$ is a metric ball with radius bounded by a
  fixed constant and center in $\Omega$:
in this case the constant $C$ in
  \eqref{eqn:regularitycondition} can be taken independent of
  $\gensub.$ Similarly, all constants appearing in the inequalities
  in Lemma~\ref{lemma-improve} are uniformly bounded on
  $\gensub,$ when $w\in A_p(\Omega,\rho,dx)$ and $\gensub$ is a metric
  ball with radius bounded by a fixed constant and center in $\Omega$.
\end{remark}

\begin{remark} \label{remark:localprop} Note that if  $w\in
  A_p(\Omega,\rho,dx)$ then  $w$ is locally doubling.  
This follows from the last inequality in  Lemma~\ref{lemma-improve} and Remark~\ref{remark:constants}.
\end{remark}

\begin{lemma} \label{lemma:rh-ap-combo} Let $\gensub\subset \R^n$ and
  suppose that $(\gensub,\rho,dx)$ is a space of homogeneous type.  If
  $w\in A_p(\gensub,\rho,dx) \cap RH_t(\gensub,\rho,dx)$, then $w^t
  \in A_q(\gensub,\rho,dx)$, where $q=t(p-1)+1$.
\end{lemma}

This result is proved in the classical setting in~\cite{MR1018575}.   The proof is the same in the setting of
spaces of homogeneous type.

\section{The machinery of PDEs}
\label{section:PDE}

Though we are ultimately interested in degenerate equations where the
degeneracy has the particular form of~\eqref{eqn:ellipticityA}
and~\eqref{eqn:ellipticityB}, we need to work in a more general
setting.  Throughout this section, $\Omega\subset \R^n$
  is open, bounded and connected, or $\Omega=\R^n$; 
  $X=(X_1,\ldots,X_m)$ will be a family of $C^\infty$ vector fields
  defined in $\R^n$ and satisfying H\"ormander's condition on every
  bounded subset of $\R^n$ if $\Omega$ is bounded, or generating a Carnot group if $\Omega=\R^n$;
  $\rho$ is the corresponding Carnot-Carath\'eodory metric in $\R^n$;
  and $Q$ is the homogeneous dimension of $\Omega$ with
  respect to $X$.   By
$\lip(\overline{\Omega})$ we mean the collection of Lipschitz
functions on $\overline{\Omega}$ and by $\lip_0(\Omega)$ the
collection of Lipschitz functions whose support is compactly contained
in $\Omega$.

\subsection{Admissible weights}
\label{subsec:soln}

We begin with a definition of the general class of weights we will be considering.

\begin{definition} \label{defn:admissible}
Given a pair of weights $(w,v)$, $w\leq v$, $v,w\in L^1_{\text{loc}}(\R^n),$ and $p$, $1<p<\infty$, we
say that the pair is $p$-admissible in $\Omega$ if:

\begin{enumerate}
\item\label{cond1admissible} $v$ is locally doubling in $(\Omega,\rho,dx)$ and $w\in A_p(\Omega,\rho,dx)$.

\item\label{cond2admissible} Given a compact set $K\subset \Omega$, there
  exists $q>p$ and $r_0>0$ such that if $B$ is a ball with center in $K$ and
  $r(B)<r_0$, then  $B\subset\Omega$ and 
$$\frac{r(B_1)}{r(B_2)}\left(\frac{v(B_1)}{v(B_2)}\right)^{1/q}\leq
C\left(\frac{w(B_1)}{w(B_2)}\right)^{1/p}$$
for all balls  $B_1\subseteq B_2\subseteq B$.
\end{enumerate}
\end{definition}

\begin{remark} 
Our definition generalizes to the two-weight case the definition of
$p$-admissible weight given in~\cite{MR1207810}.
\end{remark}

The importance of this class is that these conditions imply a two-weight Poincar\'e inequality.  When $X=\nabla$, the following result is due to Chanillo and Wheeden~\cite{MR805809}; in the general case it is due to Franchi, Lu and Wheeden~\cite{MR1343563}.

\begin{theorem} \label{thm:wtd-poincare}
Given $p$, $1<p<\infty$, a pair of $p$-admissible weights $(w,v)$ in $\Omega$, and
a compact set $K\subset \Omega$, there exists $r_0>0$ (depending on
 $K$ and $X$) such that if $B$ is a ball with center in
$K$ and $r(B)<r_0$, then for all $f\in
\lip(\overline{B})$, 
\begin{equation}\label{eqn:wtd-poincare}
\left(\frac{1}{v(B)}\int_B |f-f_{B,v}|^qv\,dx\right)^{1/q}\leq 
Cr\left(\frac{1}{w(B)}\int_B|Xf|^pw\,dx\right)^{1/p},
\end{equation}
where $f_{B,v} = \frac{1}{v(B)}\int_B f(y)v(y)\,dy$ and the constant $C$ depends only on $K,$ $\Omega,$ $X$ and the constants in conditions \eqref{cond1admissible} and \eqref{cond2admissible} in Definition \ref{defn:admissible}.
\end{theorem}

We will need a local Sobolev inequality that is a corollary of
Theorem~\ref{thm:wtd-poincare}. The proof uses an argument similar to
that in \cite[p.~80, Theorem~13.1]{MR1683160}.

\begin{corollary}
\label{coro:sobolev}
With the same hypotheses as Theorem~\ref{thm:wtd-poincare},
there exists $R _0>0$ such that if $B$ is a ball with center in
$K$ and $r(B)<R_0$, then for all $f\in
\text{Lip}_0({B})$,
\begin{equation}\label{eqn:sobolev}
\left(\frac{1}{v(B)}\int_B |f|^qv\,dx\right)^{1/q}\leq
Cr\left(\frac{1}{w(B)}\int_B|Xf|^pw\,dx\right)^{1/p}.
\end{equation}
\end{corollary}

In order to prove existence of weak solutions to $\Lap_pu=0$
(Theorem~\ref{existence}) we will need to assume a Sobolev inequality
on $\Omega$: for all $f\in \lip_0(\Omega)$,
\begin{equation}\label{eqn:global-sobolev}
\left(\int_\Omega |f(x)|^p v(x)\,dx \right)^{1/p} \leq C
\left(\int_\Omega |Xf(x)|^p w(x)\,dx \right)^{1/p}.
\end{equation}

There are several situations where this ``global'' Sobolev inequality
holds. When $X=\grad$, there are no
restrictions on the size of the balls for which the Poincar\'e
inequality holds, and so~\eqref{eqn:global-sobolev} is true for any
bounded $\Omega$.  
More generally, the global Sobolev inequality \eqref{eqn:global-sobolev} is
a  consequence of Corollary~\ref{coro:sobolev} when
$\Omega$ is such that a Hardy-type inequality is satisfied.

\begin{corollary}\label{coro:globalsobolev}
  Let $1<p<\infty$ and consider a pair of $p$-admissible weights
  $(w,v)$ in a neighborhood of $\Omega.$ If $\Omega$ is a bounded
  domain such that
\begin{equation}\label{hardy}
  \int_{\Omega}\frac{|f|^p}{\dist_\rho(x,\partial \Omega)^p}w\,dx\leq C_\Omega\int_{\Omega}|X(f)|^pw\,dx,\quad f\in \text{Lip}_0(\Omega),
\end{equation}
 then \eqref{eqn:global-sobolev} holds for 
$f\in \text{Lip}_0(\Omega)$.
\end{corollary}

We omit the proof of Corollary~\ref{coro:globalsobolev} because it is in the spirit of
the proof of \cite[Theorem 4.3]{MR2508841}. It requires a Whitney
decomposition for $\Omega$ in terms of metric balls, and a
corresponding partition of unity.
We stress that the global Sobolev inequality
\eqref{eqn:global-sobolev} will only be assumed in
Theorem~\ref{existence} and is not needed for the proof of any of our  other results.

\subsection{The spaces $S^p(\Omega)$ and $S^p_0(\Omega)$}

We now define the spaces in which weak solutions to our equations
live.    For simplicity, we state our definition assuming that
$\Omega$ is bounded.  If $\Omega=\R^n$, then we have to modify our definition by everywhere
replacing $\lip(\overline{\Omega})$ with $\lip_0(\R^n)$. 

In  this section we assume that  $(w,v)$ is a pair of $p$-admissible weights and $A=A(x)$
is a real, locally integrable, $m\times m$ symmetric matrix defined in $\Omega$ that satisfies the degenerate
ellipticity condition
\begin{equation} \label{eqn:elliptic}
w(x)^{2/p}|\xi|^2\leq \langle A(x)\xi,\xi\rangle \leq
v(x)^{2/p}|\xi|^2 \qquad \xi\in \R^m, \,  x\in \Omega.
\end{equation}
We define the space $L^p_A(\Omega)$ to consist of all measurable
vector valued functions $\f$ such that
\[ \|\f\|_A = \left(\int_\Omega \langle A\f,\f\rangle^{p/2} \,dx \right)^{1/p} <
\infty.  \]
Since $A$ is symmetric and positive semi-definite, $A^{1/2}$ exists; hence,
\[ \langle A\f,\f\rangle^{1/2} = |A^{1/2} \f|. \]
It follows immediately from this identity that $\|\cdot\|_A$ is a
norm.  In~\cite{MR2574880} (see
also~\cite{monticelli-rodney-wheedenP})  it was shown that with this norm
$L_A^p(\Omega)$ is a Banach space of equivalence classes of functions
such that $\|\f-\mathbf{g}\|_A =0$.  However, by the ellipticity
condition, we have that 
\[ \|\f\|_A \geq \left(\int_\Omega |\f|^pw\,dx\right)^{1/p}; \]
since $w\in A_p$ it is non-zero almost everywhere,  so
$\|\f\|_A=0$ if and only if $\f=\mathbf{0}$ almost everywhere.  Hence,
we can identify the elements of $L_A^p(\Omega)$ as functions defined
up to a set of measure $0$. 

If $\Omega$ is  bounded and $u\in \lip(\overline{\Omega})$, then $X u\in L^p_A(\Omega)$:
by the ellipticity condition and since $v\in L^1_{loc}(\R^n),$
\[ \|X u\|_A \leq \left(\int_\Omega |X u|^p v\,dx\right)^{1/p}
\leq \|Xu\|_\infty v(\Omega)^{1/p}< \infty.\]
%
Therefore, we can define the space $S^p(\Omega)$ to be the closure of
$\lip(\overline{\Omega})$ with respect to the norm
\[ \|u\| = \|u\|_{L^p(v,\Omega)} + \|X u\|_A. \]
Properly, $S^p(\Omega)$ is a Banach space that consists of equivalence
classes of Cauchy sequences $\{u_j\}$ of Lipschitz functions.
However, this sequence converges to a function $u\in L^p(v,\Omega)$, and the
sequence $\{X u_j\}$ converges to a function $\mathbf{U}\in L^p_A(\Omega)$
and the map $\{u_j\} \mapsto (u,\mathbf{U})$ is an isomorphism
  between $S^p(\Omega)$ and a closed subset of $L^p(v,\Omega)\times
  L^p_A(\Omega)$.  

Furthermore, if $u$ is the first element of such a
  pair, then $\mathbf{U}$ is uniquely determined.   Suppose to the
  contrary that $(u,\mathbf{U})$ and $(u,\mathbf{V})$ are in the
  image.  This is equivalent to saying that there exists a Cauchy
  sequence $\{u_j\}$ in $S^p(\Omega)$ such that $u_j \rightarrow 0$ in
  $L^p(v,\Omega)$ and $X u_j \rightarrow \mathbf{U}\neq 0$ in
  $L^p_A(\Omega)$.  By the ellipticity conditions and since $w\leq v$,
  the same limits hold in $L^p(w,\Omega)$.  However, since $w\in
  A_p(\Omega,\rho,dx)$, this is impossible.    Let $B\subset \Omega$
  be any ball and let $g\in
  C_0^\infty(B)$.  Then
\begin{align*}
\left|\int_B g\mathbf{U}\,dx \right| 
& \leq \left|\int_B g(\mathbf{U} -X u_j)\,dx\right| + \left|\int_B gX u_j\,dx\right| \\
& \leq \left(\int_B |\mathbf{U}-X u_j|^p w\,dx\right)^{1/p} \left(\int_B |g|^{p'} w^{1-p'}\,dx\right)^{1/p'} \\
& \qquad + \left(\int_B |u_j|^p w\,dx\right)^{1/p} \left(\int_B |X g|^{p'} w^{1-p'}\,dx\right)^{1/p}.
\end{align*}
By assumption, the first integral in each term goes to $0$; since
$w^{1-p'}\in A_{p'}(\Omega,\rho,dx)$, and $g,\,Xg\in L^\infty$, the second integral in each term
is finite.  Therefore, taking the limit, we get that the initial
integral is $0$; since this holds for all such $B$ and $g$,
$\mathbf{U}=\mathbf{0}$.  

\begin{remark}
The uniqueness of this representation of a Cauchy sequence $\{u_j\}$
by the pair $(u,{\bf U})$ depends strongly on the fact that $w\in
A_p(\Omega,\rho,dx)$.  In general, this need not be the case, as is
shown by the example in~\cite[p.~91]{MR643158}).
\end{remark}

Given $(u,{\bf U})$  as above with ${\bf U}=(U_1,\ldots,U_m)$, then $u$ is a
distribution and ${\bf U}$ corresponds to the
distributional $X_i$-derivative of $u$:  that is,
\begin{equation}\label{eqn:distderiv}
\int_\Omega U_i \vp\,dx=\int_\Omega u X_i^*\vp\,dx \qquad \vp\in C_0^\infty(\Omega).
\end{equation}
This follows from a calculation similar to that used to prove the
uniqueness of ${\bf U}$.  Indeed, we first note that both of the
integrals appearing in \eqref{eqn:distderiv} exist: after multiplying
and dividing by $w,$ use H\"older's inequality, and that $u,\,{\bf
  U}\in L^p(w,\Omega),$ $\varphi$ and $X_i^*\varphi$ are bounded, and
$w^{1-p'}\in L^1_{\text{loc}}(\Omega).$ To prove the equality, let
$\{u_j\}$ be a Cauchy sequence of Lipschitz functions in $S^p(\Omega)$
associated to $(u,{\bf U})$. Then
$$\int_\Omega |u_j-u|^pw\,dx\rightarrow 0 \quad \text{and} \quad \int_\Omega |{\bf U}-Xu_j|^pw\,dx\rightarrow 0,$$
and if $\vp\in C_0^\infty(\Omega)$ with $\supp(\vp)=D$, 
\begin{align*}
\Big|\int_\Omega uX_i^*\vp-U_i\vp\,dx \Big| 
& = \Big|\int_{D} (u - u_j)X^*_i \vp + \int_D (U_i-X_iu_j)\vp \,dx\Big| \\
& \leq \|X^*_i\vp\|_\infty \left(w^{1-p'}(D)\right)^{1/p'}\left(\int_D |u- u_j|^p w\,dx\right)^{1/p}\\
&\qquad+ \left(w^{1-p'}(D)\right)^{1/p'}\|\vp\|_\infty\left(\int_D |U_i-X_iu_j|^{p} w\,dx\right)^{1/p},
\end{align*}
and the last terms go to zero as $j\ra \infty$.

Hereafter, given a pair $(u,\mathbf{U})$ representing a Cauchy
sequence of Lipschitz functions in $S^p(\Omega),$ we use $X u$ instead
of $\mathbf{U}$ and, by a small abuse of notation, we write $u\in
S^p(\Omega)$.

We can also see that $S^p(\Omega)$ is a
reflexive Banach space, a fact we will need below in the proof of Theorem~\ref{existence}.   Given
$(u,\mathbf{U})\in L^p(v,\Omega)\times L^p_A(\Omega)$, the map
$T(u,\mathbf{U})=(u,A^{1/2}\mathbf{U})$ is an isometry into
$L^p(v,\Omega)\times L^p(\Omega)$ which is a reflexive space since
$1<p<\infty$.   Identifying $S^p(\Omega)$ with its image, we have that
$T(S^p(\Omega))$ is a closed subspace of a reflexive Banach space; it
follows that $S^p(\Omega)$ is also reflexive.  

We now define the space $S^p_0(\Omega)$ to be the closure of
$\lip_0({\Omega})$ in $L^p_A(\Omega)$:  more precisely,
$S^p_0(\Omega)$ consists of the equivalence classes of sequences
$\{u_j\}$ such that $\{Xu_j\}$ converges in $L^p_A(\Omega)$.  By an
abuse of notation, we will denote the Cauchy sequences with a function
$u$.

\begin{remark} If the global
Sobolev inequality~\eqref{eqn:global-sobolev} holds, then
$S^p_0(\Omega)\subset S^p(\Omega).$  To see this,  fix a Cauchy sequence
$\{u_j\}$ of an equivalence class in $S^p_0(\Omega)$.  Then this sequence
is Cauchy in $L^p(v,\Omega)$: by~\eqref{eqn:global-sobolev} and
the ellipticity condition,
\begin{multline*}
 \left(\int_\Omega |u_j-u_k|^p v\,dx\right)^{1/p}
\leq C \left(\int_\Omega |X(u_j-u_k)|^p w\,dx\right)^{1/p} \leq C\|X(u_j-u_k)\|_A.
\end{multline*}
Thus $\{u_j\}$ is a Cauchy sequence with respect to the norm in $S^p(\Omega)$ and the inclusion follows.

If the global Sobolev inequality does not hold, we may not have that
$S^p_0(\Omega)\subset S^p(\Omega).$   However, except for
Theorem~\ref{existence}, we do not
require this inclusion to hold.  (See
Section~\ref{sec:weaksolutions} for details.)
\end{remark}

\subsection{Weak solutions} \label{sec:weaksolutions}

In  this section we again assume that  $(w,v)$ is a pair of $p$-admissible weights in $\Omega$ and  let $A=A(x)$
be an $m\times m$ real symmetric matrix defined in $\Omega$ that satisfies the degenerate
ellipticity condition \eqref{eqn:elliptic}.  

To define  weak solutions, we need a functional that
acts in some ways like an inner product on $S^p(\Omega)$.   Given
$u,\vp\in \lip(\overline{\Omega})$, define 
$$a_0^p(u,\vp)=\int_\Omega \langle AXu,Xu\rangle^{\frac{p-2}{2}}\langle
AXu,X\vp\rangle\,dx.$$
This quantity is finite for all such $u$ and $\vp$.  Since $A$
is symmetric we have the inequality 
\[ |\langle A \xi,\eta\rangle|\leq \langle
A\xi,\xi\rangle^{1/2}\langle A\eta,\eta\rangle^{1/2}.  \]
Hence, by H\"older's inequality,
\begin{align}
|a_0^p(u,\vp)|
&\leq \int_\Omega \langle AXu,Xu\rangle^{\frac{p-1}{2}}\langle
AX\vp,X\vp\rangle^{1/2}\,dx \notag \\
&\leq \Big(\int_\Omega \langle AX\vp,X\vp\rangle^{p/2}\,dx\Big)^{1/p}
\Big(\int_\Omega \langle AXu,Xu\rangle^{p/2}\,dx\Big)^{1/p'} \notag\\
&=\|X\vp\|_A\|Xu\|_A^{p-1}. \label{eqn:finite}
\end{align}

We can extend $a_0^p$ to functions $u,\vp$ in $S^p(\Omega)$ or in $S_0^p(\Omega)$.

\begin{lemma}\label{lemma:convergence}
  If $u,\,\varphi\in S^p(\Omega)$ are the limits of the Cauchy
  sequences $\{u_k\}_{k\in\N}\subset \text{Lip}(\overline{\Omega})$
  and $\{\varphi_k\}_{k\in\N}\subset \text{Lip}(\overline{\Omega})$,
 then
\[
a_0^p(u,\vp) := \lim_{k\rightarrow \infty} a_0^p(u_k,\vp_k) =
\lim_{k\rightarrow\infty}\int_\Omega \langle AXu_k,Xu_k\rangle^{\frac{p-2}{2}}\langle AXu_k,X\vp_k\rangle dx
\]
is well defined:  the limit exists and does not depend on the choice
of Cauchy sequences.   The same conclusion holds if $u$ or $\varphi$
are in $S_0^p(\Omega)$.  

\end{lemma}

\begin{proof}
We prove this for $u,\,\varphi\in S^p(\Omega)$; the proof for
$S_0^p(\Omega)$ is identical. 
Fix Cauchy sequences $\{u_k\}$ and $\{\varphi_k\}$ in
$\text{Lip}({\overline\Omega})$; we will prove that
$\{a_0^p(u_k,\varphi_k)\}$ is a Cauchy sequence in $\R$.  By the
definition,
\begin{align*}
&|a_0^p(u_k,\varphi_k)-a_0^p(u_m,\varphi_m)|\\
&\quad =\bigg|\int_\Omega \langle AXu_k,Xu_k\rangle^{\frac{p-2}{2}}\langle AXu_k,X\vp_k\rangle \\
& \qquad \qquad  -  \langle AXu_m,Xu_m\rangle^{\frac{p-2}{2}}\langle AXu_m,X\vp_m\rangle dx\bigg|\\
&\quad \le  \int_\Omega  \langle AXu_k,Xu_k\rangle^{\frac{p-2}{2}}\left|\langle AXu_k,X(\vp_k-\vp_m)\rangle\right|\,dx\\
&\qquad \qquad +\int_\Omega\left|\langle\langle
  AXu_k,Xu_k\rangle^{\frac{p-2}{2}}AXu_k-
\langle AXu_m,Xu_m\rangle^{\frac{p-2}{2}}AXu_m,X\vp_m\rangle\right|\,dx\\
&\quad = I_1+I_2.
\end{align*}

The first term is easy to estimate:  the same computation as in \eqref{eqn:finite} shows that
\[
 I_1  \le \| Xu_k \|_A^{p-1}\|X\vp_k-X\vp_m\|_A,
\]
and so $I_1\to 0$ as $k,\,m\to \infty.$

The estimate for $I_2$ depends on the size of $p.$  If $p\ge 2,$ we
have the vector inequality
\begin{equation*}\label{ineqplarger2}
\left||\xi|^{p-2}\xi-|\eta|^{p-2}\eta\right|\le (p-1)\left(|\xi|^{p-2}+|\eta|^{p-2}\right)|\xi-\eta|,
\end{equation*}
where $ \xi,\,\eta\in\R^m$.  (This is implicit in~\cite[Chapter 10, p.~73]{MR2242021}.) Then, since 
 $|A^{\frac{1}{2}}Xu|=\langle AXu,Xu\rangle^{\frac{1}{2}}$, we have that for any
 $u\in \text{Lip}(\overline{\Omega})$, 
 \begin{align*}
I_2&=\int_\Omega\left|\langle\langle AXu_k,Xu_k\rangle^{\frac{p-2}{2}}AXu_k-\langle AXu_m,Xu_m\rangle^{\frac{p-2}{2}}AXu_m,X\vp_m\rangle\right|\,dx\\
&= \int_\Omega\left|\langle|A^{\frac{1}{2}}Xu_k|^{p-2} A^{\frac{1}{2}}Xu_k-|A^{\frac{1}{2}}Xu_m|^{p-2} A^{\frac{1}{2}}Xu_m,A^{\frac{1}{2}}X\vp_m\rangle\right|\,dx\\
&\le\int_\Omega  \left||A^{\frac{1}{2}}Xu_k|^{p-2} A^{\frac{1}{2}}Xu_k-|A^{\frac{1}{2}}Xu_m|^{p-2} A^{\frac{1}{2}}Xu_m\right| \left|A^{\frac{1}{2}}X\vp_m\right|\,dx\\
&\leq C \int_\Omega  \left(|A^{\frac{1}{2}}Xu_k|^{p-2} +|A^{\frac{1}{2}}Xu_m|^{p-2} \right) \left|A^{\frac{1}{2}}Xu_k-A^{\frac{1}{2}}Xu_m\right| \left|A^{\frac{1}{2}}X\vp_m\right|\,dx.
\end{align*}
Thus, by H\"older's inequality, 
\begin{align*}
I_2&\leq C \left(\int_\Omega (|A^{\frac{1}{2}}Xu_k|^{p-2} 
+|A^{\frac{1}{2}}Xu_m|^{p-2} )^{\frac{p}{p-2}}\right)^{\frac{p-2}{p}} 
\left(\int_\Omega  |A^{\frac{1}{2}}X(u_k-u_m)|^p\,dx\right)^{\frac{1}{p}}\\
&\quad\times\left(\int_\Omega
  |A^{\frac{1}{2}}X\vp_m|^{p}\,dx\right)^{\frac{1}{p}}\\
&\leq C (\|Xu_k\|_A^{p-2} +\|Xu_m\|_A^{p-2})\,\|Xu_k-Xu_m\|_A \|X\vp_m\|_A,
\end{align*}
from which it follows that $I_2\to 0$ as $k,\,m\to\infty.$

In order to estimate $I_2$ when $1<p< 2,$ we have a similar vector inequality,
\begin{equation*}\label{ineqplarger2}
\left||\xi|^{p-2}\xi-|\eta|^{p-2}\eta\right|\le c_p \, |\xi-\eta|^{p-1},
\end{equation*}
where $ \xi,\,\eta\in\R^m$.  (See~\cite[p.~43]{MR2242021}.)  We can now proceed as in the case $p\ge
2$ to get
\begin{align*}
I_2 &\leq C  \int_\Omega  \left|A^{\frac{1}{2}}Xu_k-A^{\frac{1}{2}}Xu_m\right|^{p-1} \left|A^{\frac{1}{2}}X\vp_m\right|\,dx\\ 
&\le C\,\|Xu_k-Xu_m\|_A^{p-1} \|X\vp_m\|_A.
\end{align*}

Finally, essentially the same argument shows that if $\{\tilde{u}_k\}$
and $\{\tilde{\vp}_k\}$ are also sequences converging to $u$ and
$\vp$, then 
$$\lim_{k\to \infty}a_0^p(u_k,\vp_k)=\lim_{k\to
  \infty}a_0^p(\tilde{u}_k,\tilde{\vp}_k).$$
This completes the proof.
\end{proof}

We can now define a weak solution to the equation $\Lap_pu=0$.  

\begin{definition} We say $u\in S^p(\Omega)$ is a weak solution of
  $\Lap_pu=0$ if for all $\vp\in S_0^p(\Omega)$, 
  $a_0^p(u,\vp)=0$.  
\end{definition}

To state the Harnack inequality below we need a notion of a
non-negative solution.  We will say that $u\in S^p(\Omega)$ (or $S^p_0(\Omega)$) is
non-negative, and write $u\geq 0$, if there exists a Cauchy sequence
$\{u_j\}$ in $\lip(\overline{\Omega})$ converging to $u$ such that $u_j(x)\geq 0$ for
all $x$ and $j$.  Given this, we say that $u\in S^p(\Omega)$ is a
weak subsolution of $\Lap_pu\leq 0$  if $a_0^p(u,\vp)\leq 0$ for all $\vp\in
S_0^p(\Omega)$, $\vp \geq 0$.

 Clearly, our definition of non-negative implies that $u(x)\geq 0$ almost
everywhere.  Given $u \in S^p(\Omega)$ such that $u(x)\geq 0$ a.e., we
do not know {\em a priori} that $u\geq 0$ in the sense we have just
defined, and in \cite{MR847996,MR2228656} this was taken as a
hypothesis in order to prove regularity results.  However, if we impose a stronger condition on the weight
$v$, then we do get that our definition coincides with the usual
definition of non-negative.  

\begin{prop} \label{prop:positive} Given a ball $B$ and $v\in
  A_\infty(B,\rho,dx)$ suppose that $\phi\in S^p(B)$ is such that
  $\phi(x)\geq 0$ a.e. in $B$.  Then $\phi\geq 0$: that is, there exists a
  representative Cauchy sequence $\{ \psi_k\}$ of $\phi$ such that
  $\psi_k(x) \geq 0$  for all $k$.  
\end{prop}

\begin{proof}
Let $\{\phi_k\} \subset \lip(\overline{\Omega})$ be a representative
Cauchy sequence of $\phi$.  Define the sequence $\{\psi_k\}$ by 
\[ \psi_k(x) = \max\big(\phi_k(x),-1/k\big)+ 1/k. \]
Then for all $k$, $\psi_k(x)\geq 0$ and $\psi_k \in
\lip(\overline{\Omega})$.  It will suffice to prove that there exists
a subsequence $\{\psi_{k_j}\}$ such that
\[ \lim_{j\rightarrow \infty} \|\phi_{k_j}-\psi_{k_j}\| = 0; \]
it follows from this that $\{\psi_{k_j}\}$ is a representative Cauchy
sequence for $\phi$.  Since we will have to pass to a subsequence
several times, we will abuse notation and denote each
successive subsequence by $\{\psi_k\}$.

We first construct a subsequence such that $\|\psi_k-\phi_k\|_{L^p(v,B)}\rightarrow 0$.  Let $E_k = \{ x\in B : \phi_k(x)<-1/k\}$.  Then 
\[ \int_B |\psi_k(x)-\phi_k(x)|^p v(x)\,dx = k^{-p}v(B\setminus E_k) +
\int_{E_k} |\phi_k(x)|^p v(x)\,dx. \]
The first term clearly tends to zero as $k\rightarrow \infty$.  Since
$\phi_k\rightarrow \phi$ in $L^p(v,B)$ norm, by the converse of
the dominated convergence theorem (see~\cite{MR1817225}), there exists $\Phi\in L^1(v,B)$ such
that, after we pass to a subsequence, $\phi_k\rightarrow \phi$
pointwise a.e. and $|\phi_k|^p \leq \Phi$.
Therefore, by the dominated convergence theorem, to show that the
second term above tends to $0$, it will suffice to show that
$|E_k|\rightarrow 0$ as $k\rightarrow 0$.

Since $\phi(x)\geq 0$ a.e., we have that 
\[ E_k = \{ x\in B : \phi_k(x) < -1/k \} \subset 
\{ x\in B : |\phi_k(x)-{\phi(x)}|>1/k \} = F_k. \]
Since $\phi_k \rightarrow \phi$ in $L^p(v,B)$, it converges in
measure with respect to the measure $v\,dx$, so by again passing to a
subsequence we may assume that $v(F_k)<1/k$.  

By assumption $v\in A_\infty(B,\rho,dx)$, and so by Lemma~\ref{lemma-improve},
\[ \frac{|F_k|}{|B|} \leq C\left(\frac{v(F_k)}{v(B)}\right)^{1/p}. \]
Hence, $|F_k|\rightarrow 0$ as $k\rightarrow \infty$, and so
$|E_k|\rightarrow 0$ as well.

We will now show that there exists a subsequence such that
\[  \lim_{k\rightarrow \infty} \|X\psi_k-X\phi_k\|_A^p =\lim_{k\rightarrow \infty} \int_B \langle
AX(\psi_k-\phi_k),X(\psi_k-\phi_k)\rangle^{p/2}\,dx = 0. \]
If $f=c$ is constant, then $Xf=0$ a.e.  Hence, on $B\setminus E_k$,
$X\phi_k=X\psi_k$, and so 
by dividing the domain as
we did above, we have to show that 
\begin{equation} \label{eqn:positive1}
\lim_{k\rightarrow \infty} \int_{E_k} \langle
AX\phi_k,X\phi_k\rangle^{p/2}\,dx = 0.
\end{equation}
We will actually construct a sequence of nested sets $G_k$ such that
$E_k\subset G_k$ and show that 
\begin{equation} \label{eqn:positive2}
\lim_{k\rightarrow \infty} \int_{G_k} \langle
AX\phi_k,X\phi_k\rangle^{p/2}\,dx = 0. 
\end{equation}
Since $\langle AX\phi_k,X\phi_k\rangle =|A^{1/2}X\phi_k|^2$,
the integrand is positive and so~\eqref{eqn:positive1} follows
from~\eqref{eqn:positive2}.

Since $|E_k|\rightarrow 0$, by passing to a subsequence, we may assume
that $|E_{k+1}|\leq \frac{1}{2}|E_k|$.  Define
\[ G_k = \bigcup_{j\geq k } E_j. \]
Then $|G_k| \leq 2|E_k|$, and so $|G_k|\rightarrow 0$ as $k\rightarrow
\infty$. 

Since $\phi_k$ is Cauchy with respect to the $S^p(B)$ norm,  we know
that the sequence of functions $f_k = \langle AX\phi_k,X\phi_k\rangle^{p/2}$
is Cauchy in $L^1(B)$ and therefore converges to some function $f\in
L^1(B)$.  Fix $\epsilon>0$.  Then there exists $k_0$ such that 
\[ \|f\chi_{G_{k_0}}\|_{L^1(B)} < \epsilon/2. \]
But for fixed $k_0$, we have that $f_k\chi_{G_{k_0}} \rightarrow
f\chi_{G_{k_0}}$ in $L^1(B)$.  Hence, for all $k\geq k_0$
sufficiently large, we have that 
\[  \|f_k\chi_{G_{k}}\|_{L^1(B)} \leq  \|f_k\chi_{G_{k_0}}\|_{L^1(B)}
< \epsilon. \]
Since $\epsilon$ is arbitrary, we get \eqref{eqn:positive2} and our
proof is complete.
\end{proof}

Finally, we show that weak  solutions to $\Lap_pu=0$ exist.  When $p=2$ this is a
straightforward application of the Lax-Milgram theorem.  For all $p>1$
we need a more delicate argument.

\begin{theorem}\label{existence}

Suppose $(w,v)$ is a pair of  weights satisfying condition 
 \eqref{cond1admissible}  of Definition~\ref{defn:admissible} such that the
  global Sobolev inequality~\eqref{eqn:global-sobolev} holds. Let   $A=A(x)$
be an $m\times m$ real symmetric matrix defined in $\Omega$ that satisfies the degenerate
ellipticity condition \eqref{eqn:elliptic}.
 Then  for all $\psi\in S^p(\Omega)$, there exists $u\in S^p(\Omega)$ such
  that $u$ is a weak solution of $\Lap_pu=0$ in $\Omega$ and $u-\psi\in
  S_0^p(\Omega)$.
\end{theorem}

\begin{remark}
  The global Sobolev inequality \eqref{eqn:global-sobolev} is only
  assumed in Theorem~\ref{existence} and we do not use it in the proof of
  any of our other results. For situations when it holds see discussion following
  \eqref{eqn:global-sobolev}.
\end{remark}

For the proof of Theorem~\ref{existence} we will use an abstract
result taken from Kinderlehrer and Stampacchia~\cite{MR1786735}.  Let
$V$ be a reflexive Banach space with norm $\|\cdot\|$ and dual
$V'$, and denote by $(\cdot,\cdot)$ a pairing between $V'$ and $V.$
Given a nonempty closed convex subset $U\subset V$ and a mapping
$\mathcal{A}:U\to V'$, $\mathcal{A}$ is monotone if
\[
(\mathcal{A}u_1-\mathcal{A}u_2,u_1-u_2 )\ge 0,\quad  u_1,\,u_2 \in U, 
\]
and $\mathcal{A}$ is coercive if there exists $g\in U$ such that
\[
\frac{(\mathcal{A}u_j-\mathcal{A}g,u_j-g )}{\|u_j-g\|}\to \infty
\]
 for all sequences $\{u_j\}\subset U$ such that $\|u_j\|\to\infty.$

\begin{prop}[see \cite{MR1786735},  page 87]\label{abstractexistence}
Let $U$ be a nonempty  closed convex subset of a reflexive Banach
space $V$ and let $\mathcal{A}:U\to V'$ be monotone, coercive, and
weakly continuous on $U$.  Then there exists $u\in U$ such that 
\[
\langle\mathcal{A}u, g-u\rangle\ge 0, \quad \forall g\in U.
\]   
\end{prop}

\begin{proof}[Proof of Theorem~\ref{existence}] With the notation of
  Proposition~\ref{abstractexistence} we set $V=S^p(\Omega),$
  $U=\{u\in S^p(\Omega): u-\psi\in S^p_0(\Omega)\},$ and define
  $\mathcal{A}u(\cdot)=a_0^p(u,\cdot)$ for $u\in U$.  Recall that
  $S^p(\Omega)$ is a reflexive Banach space.  Furthermore, it is clear
  that $U$ is a nonempty (since $\psi\in U$),   
  convex closed subset of $S^p(\Omega)$. We claim that 
  $\mathcal{A}u\in S^p(\Omega)'$ for all $u\in U$ (in fact for all
  $u\in S^p(\Omega)$).   First, $a_0^p(u,\cdot)$ is clearly
  linear.  Further, by using~\eqref{eqn:finite} for  functions in
  Lip$(\overline{\Omega})$ and then passing to the limit using
  Lemma~\ref{lemma:convergence}, we have that for all $u,\,g\in S^p(\Omega)$,
\[
|a_0^p(u,g)|\le \|u\|^{p-1}\|g\|.
\]

To prove that $\mathcal{A}$ is monotone, let $u_1,\,u_2\in
\lip(\overline{\Omega})$; then a straightforward computations gives
\begin{align*}
&(\mathcal{A}u_1-\mathcal{A}u_2,u_1-u_2 )=a_0^p(u_1,u_1-u_2)-a_0^p(u_2,u_1-u_2)\\
&=\int_\Omega\langle|A^{\frac{1}{2}}Xu_1|^{p-2}A^{\frac{1}{2}}Xu_1-|A^{\frac{1}{2}}Xu_2|^{p-2}A^{\frac{1}{2}}Xu_2, A^{\frac{1}{2}}Xu_1-A^{\frac{1}{2}}Xu_2\rangle\,dx.
\end{align*}
Since 
\[
\langle|\xi|^{p-2}\xi-|\eta|^{p-2}\eta,\xi-\eta\rangle\ge 0, \quad \forall \xi,\,\eta\in \R^m,
\]
it follows that
\begin{equation}\label{monotone}
(\mathcal{A}u_1-\mathcal{A}u_2,u_1-u_2 )\ge 0.
\end{equation}
By Lemma~\ref{lemma:convergence} we conclude that \eqref{monotone}
holds for all $u_1,\,u_2\in S^p(\Omega)$.

To show that $\mathcal{A}$ is weakly continuous again fix
$u_1,\,u_2,\,g\in \lip(\overline{\Omega})$.  Then 
\begin{align*}
&|(\mathcal{A}u_1,g)-(\mathcal{A}u_2,g)|=a_0^p(u_1,g)-a_0^p(u_2,g)\\
&=\int_\Omega\langle|A^{\frac{1}{2}}Xu_1|^{p-2}A^{\frac{1}{2}}Xu_1-|A^{\frac{1}{2}}Xu_2|^{p-2}A^{\frac{1}{2}}Xu_2,
A^{\frac{1}{2}}Xg\rangle\,dx. 
\end{align*}
If we repeat the argument used to estimate the term $I_2$ in
Lemma~\ref{lemma:convergence}, we get 
\begin{equation*}
|(\mathcal{A}u_1,g)-(\mathcal{A}u_2,g)|\leq C (\|u_1\|^{p-2} +\|u_2\|^{p-2})\,\|u_1-u_2\| \|g\|,\quad p\ge 2,
\end{equation*}
and 
\begin{equation*}
|(\mathcal{A}u_1,g)-(\mathcal{A}u_2,g)|\leq C \|u_1-u_2\|^{p-1} \|g\|,\quad 1<p< 2.
\end{equation*}
Therefore, passing to the limit we have that these inequalities hold
for any $u_1,\,u_2,\,g\in S^p(\Omega)$.  Hence, $\mathcal{A}$ is
weakly continuous on $S^p(\Omega)$:  if
$\{u_j\}\subset S^p(\Omega)$ is such that $u_j\to u$ then
$(\mathcal{A}u_j,g)\to (\mathcal{A}u,g)$ for all $g\in S^p(\Omega)$.

Finally, we show that $\mathcal{A}$ is coercive. 
We will prove that 
\begin{equation} \label{eqn:coercive}
\frac{(\mathcal{A}u_j-\mathcal{A}\psi,u_j-\psi )}{\|u_j-\psi\|}\to \infty
\end{equation}
for all sequences $\{u_j\}\subset U$ such that $\|u_j\|\to\infty,$
 where $\psi$ is as in the statement of the theorem.   Since the
 global Sobolev inequality holds,  we have that $\|u_j-\psi\|\approx
 \|u_j-\psi\|_A$; hence, we can use $\|u_j-\psi\|_A$ in the
 denominator.   Further, by a limiting
 argument using Lemma~\ref{lemma:convergence} it suffices to prove this
 for $u,\,\psi\in \lip(\overline{\Omega})$.  

We first consider the
 case $p\geq 2$.  
Arguing as we did before, we have that 
\begin{align*}
&(\mathcal{A}u-\mathcal{A}\psi,u-\psi )=a_0^p(u,u-\psi)-a_0^p(\psi,u-\psi)\\
&=\int_\Omega\left|\langle|A^{\frac{1}{2}}Xu|^{p-2}A^{\frac{1}{2}}Xu-|A^{\frac{1}{2}}X\psi|^{p-2}A^{\frac{1}{2}}X\psi, A^{\frac{1}{2}}Xu-A^{\frac{1}{2}}X\psi\rangle\right|\,dx.
\end{align*}
Since $p\ge 2 $, we have that
\begin{equation} \label{eqn:p>2}
\langle|\xi|^{p-2}\xi-|\eta|^{p-2}\eta,\xi-\eta\rangle\ge
2^{2-p}|\xi-\eta|^p, \quad  \forall \xi,\,\eta\in \R^m.
\end{equation}
 (See~\cite[Chapter 10]{MR2242021} for a proof of \eqref{eqn:p>2}).  Hence,
\[ 
(\mathcal{A}u-\mathcal{A}\psi,u-\psi )\ge 2^{2-p}\, \int_\Omega \left|A^{\frac{1}{2}}Xu-A^{\frac{1}{2}}X\psi\right|^p\,dx
= 2^{2-p}\, \|u-\psi\|_A^p,
\]
and \eqref{eqn:coercive} follows immediately.  

When $1<p\leq 2$ we argue as before, except that we use the inequality
\begin{equation} \label{eqn:p<2}
 \langle|\xi|^{p-2}\xi-|\eta|^{p-2}\eta,\xi-\eta\rangle \geq
|\xi-\eta|^p - |\eta|^{p-1}|\xi-\eta|, \quad \forall \xi,\,\eta\in \R^m. 
\end{equation}
(Inequality \eqref{eqn:p<2} follows in a similar manner to \eqref{eqn:p>2}.)  Combining this with H\"older's inequality, we have that 
\begin{align*}
& (\mathcal{A}u-\mathcal{A}\psi,u-\psi ) \\
& \qquad \geq \int_\Omega
\left|A^{\frac{1}{2}}Xu-A^{\frac{1}{2}}X\psi\right|^p\,dx
- \int_\Omega
\left|A^{\frac{1}{2}}X\psi|^{p-1}\right|\left|A^{\frac{1}{2}}Xu-A^{\frac{1}{2}}X\psi\right|\,dx
\\
& \qquad \geq \|u-\psi\|_A^p - \|\psi\|_A^{p-1}\|u-\psi\|_A.
\end{align*}
Again, \eqref{eqn:coercive} follows immediately.  

Therefore, by Proposition~\ref{abstractexistence} there exists $u\in U$ such that 
\[
a_0^p(u,g-u)\ge 0
\]
 for all  $g\in U.$  If $\varphi\in S_0^p({\Omega}),$ then $u+\varphi$ and $u-\varphi$ belong to $U$ and therefore
\[
a_0^p(u,\varphi)\ge 0\quad \text{ and }\quad -a_0^p(u,\varphi)\ge 0,
\] 
from which it follows that $a_0^p(u,\varphi)= 0.$
Therefore, we have shown that $u$ is a weak solution of $\Lap_pu=0$ such
that $u-\psi\in S_0^p(\Omega)$.

%
%
%
%
%
%
%
%

\end{proof}

\subsection{Harnack inequality}
\label{subsec:harnack}
Central to the proofs of our main results are a mean value inequality and a Harnack
inequality for weak solutions of the equation $\Lap_p u=0$.  As
always, $A=A(x)$ is a real,  $m\times m$ symmetric
matrix defined in $\Omega$ that satisfies the degenerate ellipticity
condition \eqref{eqn:elliptic}.

\begin{lemma} \label{lemma:mvt}
Given $p$, $1<p<\infty$, a pair of $p$-admissible
weights $(w,v)$ in $\Omega$, and a compact set $K\subset \Omega$, let $r_0>0$ be
as in Theorem~\ref{thm:wtd-poincare}.  Then given any ball $B$ with
center in $K$ and $r(B)<r_0$, suppose $u\in S^p(B)$ is a weak
subsolution of $\Lap_p u \leq 0$.   Then there exist constants $c$ and
$d$, depending only on the weights, such that, for any $\alpha$,
$1/2\leq \alpha <1$, 
\[ \big( \esssup_{x\in \alpha B} u^+(x) \big)^p \leq
\frac{c}{(1-\alpha)^d} \mu_p^{\frac{p\sigma}{\sigma-1}} \frac{1}{v(B)}
  \int_B u^+(x)^p v(x)\,dx, \]
where $\sigma = q/p>1$ with $q$ as in  condition (2) of
Definition~\ref{defn:admissible}, 
$u^+=\max(u,0)$, and 
\[ \mu_p = \mu_p(B) = \left(\frac{v(B)}{w(B)}\right)^{1/p}. \]
\end{lemma}

\begin{remark}
We may restate Lemma~\ref{lemma:mvt} as follows:  given any
$x\in \Omega$, then for all balls centered at $x$ whose radius is
sufficiently small, the mean value inequality holds.
\end{remark}

Lemma~\ref{lemma:mvt} was proved in~\cite{MR2228656} for non-negative
weak subsolutions, with $u^+$ replaced by $u$.   However, the proof
can be readily adapted:  for details on the minor
changes required, see~\cite{MR847996} where the argument is done for
the case $p=2$; the same changes apply for arbitrary $p>1$.    As a
consequence, we get the following result. 

\begin{theorem} \label{thm:strong-mvt}
With the same hypotheses as Lemma~\ref{lemma:mvt}, suppose that $u\in S^p(B)$ is a weak
solution of $\Lap_p u = 0$.  Then
\[  \esssup_{x\in \alpha B} |u(x)|^p \leq
\frac{c}{(1-\alpha)^d} \mu_p^{\frac{p\sigma}{\sigma-1}} \frac{1}{v(B)}
  \int_B |u(x)|^p v(x)\,dx. \]
\end{theorem}

\begin{proof}  Since $u$ is a weak solution, then $u$ and
$-u$ are both weak subsolutions, so the conclusion of
Lemma~\ref{lemma:mvt} holds for both of these functions.  Since
\[ |u| = \max(u^+, (-u)^+), \]
the desired inequality follows at once. 
\end{proof}

Our final result is a Harnack inequality, which is also proved
in~\cite{MR2228656}.

\begin{theorem} \label{thm:harnack}
Given $p$, $1<p<\infty$, a pair of $p$-admissible
weights $(w,v)$ in $\Omega$, and a compact set $K\subset \Omega$, let $r_0>0$ be
as in Theorem~\ref{thm:wtd-poincare}.  Given any ball $B$ with
center in $K$ and $r(B)<r_0/2$, suppose $u$ is a non-negative weak
solution of $\Lap_pu=0$ in $S^p(2B)$.   Then
$$\esssup_B u \leq e^{C_K\mu_p(B)} \essinf_B u$$
with $\mu_p(B)=(v(B)/w(B))^{1/p}$ and $C_K$ independent of $u$ and $B$.
\end{theorem}

\begin{remark}
As with the previous two results, we may rephrase
Theorem~\ref{thm:harnack} as saying that given $x\in B$, the Harnack
inequality holds for all balls $B$ centered at $x$ with radius
sufficiently small.  
\end{remark}

\begin{remark}
In~\cite{MR2228656}, Theorem~\ref{thm:harnack} is proved for solutions $u$
that are non-negative in the sense that there exists a sequence
$\{u_j\}$ of non-negative Lipschitz functions converging to $u$.
By Proposition~\ref{prop:positive}, when we apply this result below it
will suffice to assume that $u$ is a non-negative in the sense that
$u(x)\geq 0$ a.e.
\end{remark}

\section{Proof of main results}
\label{proofofmain}

In this section we state and prove our main results.  
 Throughout, $\Omega\subset \R^n$
  is open, bounded and connected, or $\Omega=\R^n$; 
  $X=(X_1,\ldots,X_m)$ will be a family of $C^\infty$ vector fields
  defined in $\R^n$ and satisfying H\"ormander's condition on every
  bounded subset of $\R^n$ if $\Omega$ is bounded, or generating a Carnot group if $\Omega=\R^n$;
  $\rho$ is the corresponding Carnot-Carath\'eodory metric in $\R^n$;
  and $Q$ is the homogeneous dimension of $\Omega$ with
  respect to $X$.

Our first result is a generalization of Theorem A in the
Introduction.  

\begin{theorem} \label{thm:easy-thm}
Given $\Omega$, $X$, and $p$, $1<p<\infty$, let $A$ be an
$m\times m$ matrix of measurable functions defined in $\Omega$ that satisfies the
ellipticity condition
\begin{equation} \label{eqn:easy1}
 \lambda w(x)^{2/p}|\xi|^2 \leq \langle A(x)\xi,\xi \rangle \leq \Lambda
 w(x)^{2/p}|\xi|^2, \qquad \xi\in \R^m,\,x\in \Omega,
\end{equation}
where $w\in L^1_{\mathrm{loc}}(\R^n)$,  $w\in A_p(\Omega,\rho,dx)$ and $0<\lambda<\Lambda<\infty$. 
Then every weak solution $u$ of the equation
\begin{equation} \label{eqn:easy2}
\Lap_pu=- \Div_X(\langle AXu,Xu\rangle^{\frac{p-2}{2}} AXu) = 0
\end{equation}
is   H\"older continuous on compact
subsets of $\Omega.$
\end{theorem}

\begin{remark}
Recall that by Remark~\ref{rem:restriction} if $\Omega$ satisfies
certain boundary conditions, then we can assume $w\in
A_p(\R^n,\rho,dx)$.  Conversely, by Wolff's extension theorem, in
certain cases,  $w\in A_p(\Omega,\rho,dx)$ can be extended to an $A_p$
weight on $\R^n$.  See~\cite{garcia-cuerva-rubiodefrancia85} for more
details. 
\end{remark}

Our
next two  results generalize Theorem B from the Introduction.  We have  
two results depending on whether we are working in a general
Carnot-Carath\'eodory space or if we are in the special case of a Carnot
group.    To state both results, we need two  definitions.

First, we defined the restricted maximal operator $M^*$ by
\[ M^*f(x) =\sup_{0<r<1}
\dashint_{B(x,r)\cap \Omega^*} |f(y)|\,dy. \]
If $\Omega=\R^n$, we take $\Omega^*=\R^n$; if $\Omega$ is bounded,
then we take $\Omega^*=B^*$ to be a ball containing $\Omega$ such that
if $x\in \Omega$, $B(x,1)\subset B^*$.  In either case $(B^*,\rho,dx)$ is a space
of homogeneous type, and so if $f$ is locally integrable,
$M^*f(x)<\infty$ almost everywhere.

\begin{definition}
Given a function $u\in L^1_{loc}$, we say that $u$ has an approximate limit at a point $x$ if there exists $z\in \R$ such that 
\[ \lim_{r\rightarrow 0} \dashint_{B(x,r)} |u(y)-z|\,dy = 0.  \]
Note that by the Lebesgue differentiation theorem, this limit exists
almost everywhere and we have $z=u(x)$.  At each point $x$ where $u$
has an approximate limit, define the function $\tilde{u}(x)=z$.  The
function $\tilde{u}$ is called the precise representative of $u$.
\end{definition}

\begin{theorem} \label{thm:main-CCS}
Given $\Omega$, $X$ and $p$, $Q-n<p<\infty$, let $A$ be an
$m\times m$ matrix of measurable functions defined in $\Omega$ that satisfies the
ellipticity condition
\begin{equation} \label{eqn:CCS1}
 k(x)^{-2/p'}|\xi|^2 \leq \langle A(x)\xi,\xi \rangle \leq 
 k(x)^{2/p}|\xi|^2, \qquad \forall \xi\in \R^m,\,\text{a.e. } x\in \Omega,
\end{equation}
where $k, \, k^{-p'/p} \in L^1_{loc}(\R^n)$ and 
\[ k\in A_{p'}(\Omega,\rho,dx)\cap RH_\tau (\Omega,\rho,dx), \qquad \tau = 1+\frac{p(Q-1)}{n+p-Q}. \]
Then every weak solution
of~\eqref{eqn:easy2} is continuous almost everywhere.  More precisely,
$\tilde{u}$, the precise representative of $u$, 
is continuous at every point of the set 
\[ S(k) = \{ x\in \Omega : M^*k(x) < \infty \}. \]
\end{theorem}

When $X$ generates a homogeneous Carnot group, we have the following theorem.

\begin{theorem}\label{thm:main-carnot} 
Given $\Omega=\R^n$, $X$ and $p$, $1<p<\infty$, suppose $X$ generates
a homogeneous Carnot group.  Let $A$ be an
$m\times m$ matrix of measurable functions that satisfies the
ellipticity condition~\eqref{eqn:CCS1}, where $k, \, k^{-p'/p} \in L^1_{loc}(\R^n)$ and  $k\in A_{p'}(\Omega,\rho,dx)\cap RH_Q (\Omega,\rho,dx)$. 
Then every weak solution
of~\eqref{eqn:easy2} is continuous almost everywhere:  the precise representative $\tilde{u}$
is continuous at every point of the set 
\[S(k) = \{ x\in \R^n : M^*k(x) < \infty \}. \]
 \end{theorem}

\begin{remark}
Note that since $Q\geq n$, $\tau\geq Q$.  Therefore, the
hypotheses on the weight $k$ in Theorem~\ref{thm:main-carnot} is weaker than in
Theorem~\ref{thm:main-CCS}. 
\end{remark}

\begin{remark}
  Since $k\in RH_n$, $M^*k$ is a locally integrable function.  Hence, the complement of the set $S(k)$ has measure zero, and
hence the weak solutions to $\Lap_pu=0$ are continuous almost
everywhere.
\end{remark}

\begin{proof}[Proof of Theorem~\ref{thm:easy-thm}]
  We want to use the Harnack inequality, Theorem~\ref{thm:harnack}.  To do so, we first have to
  prove that $(\lambda^{p/2}w,\Lambda^{p/2}w)$ is a pair of
  $p$-admissible weights.  Since $w\in A_p (\Omega,\rho,dx)$, so is
  $\lambda^{p/2}w$, and by Remark~\ref{remark:localprop},
  $\Lambda^{p/2}w$ is locally doubling.  We also have that condition
  \eqref{cond2admissible} in Definition \ref{defn:admissible} holds.  Indeed, let $K$ be a compact subset
  of $\Omega$ , $R_K$ as given by Lemma~\ref{lemma:hormander-nocenter},
  $B$ a metric ball with center in $K$ and $r(B)\le R_K,$ and
  $B(x_1,r_1)\subset B(x_2,r_2)\subset B.$ By
  Lemma~\ref{lemma:hormander-nocenter} it is enough to show that there
  exists $q>p$ such that
\begin{equation} \label{eqn:easy3}
\left(\frac{|B(x_1,r_1)|}{|B(x_2,r_2)|}\right)^{1/Q} \leq 
C\left(\frac{w(B(x_1,r_1))}{w(B(x_2,r_2))}\right)^{\frac{1}{p}-\frac{1}{q}}. 
\end{equation}
By Lemma~\ref{prop-Apfacts} and Remark~\ref{remark:constants}  there exists $\epsilon>0,$ independent of $B,$ such that
$w\in A_{p-\epsilon}(B,\rho,dx)$ with uniform  constants.  Since $Q\ge n\geq 2$, we may assume that $\epsilon$
is so small that $Q> p/(p-\epsilon)$.  Therefore,  we can define $q>p$
by
\[ \frac{1}{p}-\frac{1}{Q(p-\epsilon)} = \frac{1}{q}, \]
and by Lemma~\ref{lemma-improve} (with $p-\epsilon$ in place of $p$
and $E=B(x_1,r_1)$) and Remark~\ref{remark:constants} we get
\eqref{eqn:easy3} with uniform constants.  Hence,
$(\lambda^{p/2}w,\Lambda^{p/2}w)$ is a pair of $p$-admissible weights.  

Any solution now satisfies the hypotheses of Theorems \ref{thm:strong-mvt} and \ref{thm:harnack}.  In particular, solutions are locally bounded and non-negative solutions satisfy a locally uniform Harnack inequality. We can now prove H\"older continuity. 
Let $u\in S^p(\Omega)$ be a solution of \eqref{eqn:easy2}.  Fix
$x\in \Omega$ and let $B=B(x,r)\subset \Omega$, where $r$ is
sufficiently small for Theorems~\ref{thm:strong-mvt}
and~\ref{thm:harnack} to hold, say $r\le r_0$.  Then by Theorem~\ref{thm:strong-mvt}, $u$ is bounded on $B$.  Let 
\begin{align*}
M = \esssup_{ \frac{1}{2}B} u,& \; m= \essinf_{\frac{1}{2}B} u, \quad M' = \esssup_{ B} u,  \; m'= \essinf_{ B} u.
\end{align*}

Then the functions $M'-u$ and $u-m'$ are non-negative solutions on $B$.  Therefore, by Theorem~\ref{thm:harnack} applied to the ball $ \frac{1}{2}B$, we get that
\begin{multline*}
M'-m = \esssup_{y\in \frac{1}{2}B} \big(M'-u(y)\big) \\ \leq \exp\left(C (\Lambda/\lambda)^{1/2}\right) 
\essinf_{y\in \frac{1}{2}B} \big(M'-u(y)\big) = \exp\left(C (\Lambda/\lambda)^{1/2}\right) (M'-M)
\end{multline*}
and
\begin{multline*}
M-m' = \esssup_{y\in \frac{1}{2}B} \big(u(y)-m'\big) \\ \leq \exp\left(C (\Lambda/\lambda)^{1/2}\right) 
\essinf_{y\in \frac{1}{2}B} \big(u(y)-m'\big) = \exp\left(C (\Lambda/\lambda)^{1/2}\right) (m-m').
\end{multline*}

Let $P=\exp\left(C (\Lambda/\lambda)^{1/2}\right)$.  If we add these two inequalities and rearrange terms, we get that 
\[ \osc_u(x,r/2) \leq \frac{P-1}{P+1}\osc_u(x,r), \]
where $\osc_u(x,s):=\esssup_{B(x,s)} u-\essinf_{B(x,s)}u.$
Therefore, if we iterate this inequality, by a standard argument (see, for instance,~\cite[Lemma~8.23]{MR1814364}) we get that there exist $\beta$ and $\alpha$ such that
\[
\osc_u(x,r) \leq \beta\,\osc_u(x,r_0) \left(\frac{r}{r_0}\right)^{\alpha},\quad 0<r<r_0.
\]
The constant $\beta$ and the exponent $\alpha$ depend only on $P.$ Then  $u$ is H\"older continuous on compact subsets of $\Omega.$  For more details we refer the reader to the end of the proof of Theorems \ref{thm:main-CCS} and \ref{thm:main-carnot}.
\end{proof}

The proofs of Theorems~\ref{thm:main-CCS} and~\ref{thm:main-carnot}
follow the same basic outline as the proof of
Theorem~\ref{thm:easy-thm}.  A key step in both, however, is proving
that $(k^{-p/p'},k)$ is a pair of $p$-admissible weights.   By
assumption, $k\in A_{p'}(\Omega,\rho,dx)$, so $k^{-p/p'}=k^{1-p}\in
A_p (\Omega,\rho,dx)$.   That condition $(2)$ in Definition \ref{defn:admissible} holds is the substance of the next two results.

\begin{theorem} \label{thm:hormander:balance} Let $p>Q-n.$ The pair $(k^{-p/p'},k)$
  satisfies condition \eqref{cond2admissible} of
  Definition~\ref{defn:admissible} provided that $k^{-p/p'}\in
  A_s(\Omega,\rho,dx)$, $k\in RH_t(\Omega,\rho,dx)$, $1<t<\infty$,
  $\frac{p}{Q}<s<\infty,$ and
\begin{equation}\label{eqn:hb1}
 \frac{pt'}{n} \leq \left(\frac{sQ}{p}\right)'-1.
\end{equation}
In particular, this is the case if
\begin{equation} \label{eqn:hb0}
 k \in A_{p'}(\Omega,\rho,dx) \cap RH_{\tau}(\Omega,\rho,dx), \qquad \tau =
1+\frac{p(Q-1)}{n+p-Q}.
\end{equation}
\end{theorem}

If $X$ generates a Carnot group on $\R^n$, then we can repeat the proof of Theorem~\ref{thm:hormander:balance} with $n$ replaced by $Q$ to get the following result.

\begin{theorem} \label{thm:carnot-balance} Suppose $\Omega=\R^n$ and
  $X$ generates a homogeneous Carnot group. Then the pair
  $(k^{-p/p'},k) $ satisfies condition \eqref{cond2admissible} of
  Definition \ref{defn:admissible} with uniform constants for all
  compact sets, provided that $k^{-p/p'}\in A_s(\R^n,\rho,dx)$, $k\in
  RH_t(\R^n,\rho,dx)$, $1<t<\infty$, $\frac{p}{Q}<s<\infty,$ and
\[ \frac{pt'}{Q} \leq \left(\frac{sQ}{p}\right)'-1. \]
In particular, this is the case if
\[ k \in A_{p'}(\R^n,\rho,dx) \cap RH_{Q}(\R^n,\rho,dx),\]
which in turn implies that $k^{-p/p'}\in A_{1+\frac{p-1}{Q}}(\R^n,\rho,dx)$.
\end{theorem}

\begin{proof}[Proof of Theorem~\ref{thm:hormander:balance}]
  Let $R_\Omega$ be given by Lemma \ref{lemma:hormander-nocenter}; let
  $B$ be a metric ball with center in $\Omega$ and radius $r(B)\leq
  R_\Omega$; and let $B_1=B(x_1,r_1)$, $B_2=B(x_2,r_2)$ be such that
  $B_1\subset B_2\subset B$.  Fix a pair $(k^{-p/p'},k)$ as in the
  hypotheses.  Then by \eqref{eqn:hb1},
\begin{equation} \label{eqn:hb2}
 \frac{n}{pt'} \geq \frac{sQ}{p}-1.
\end{equation}
By Lemma~\ref{prop-Apfacts}, since the restriction of $k^{-p/p'}$ to $B$ belongs to $A_s(B,\rho,dx)$, there exists
$\epsilon>0$ such that the restriction of  $k^{-p/p'}$ to $
B$ is in $A_{s-\epsilon}(B,\rho,dx)$.   Furthermore, by Remark~\ref{remark:constants}, we may choose $\epsilon$  independent of $B$ and the constants are uniform.   Combining this
with \eqref{eqn:hb2} we get that
\[ \frac{n}{p}\left(1-\frac{1}{t}\right) >
\frac{(s-\epsilon)Q}{p}-1. \]
Rearranging terms we define
\begin{equation} \label{eqn:hb3}
 q = \frac{n(t')^{-1}}{(s-\epsilon)Qp^{-1}-1} > p.
\end{equation}

Therefore, since  $k\in RH_t (B,\rho,dx)$, by Lemmas~\ref{lemma:hormander-nocenter}
and~\ref{lemma-improve} and inequality \eqref{eqn:hb3},
\begin{multline*}
\frac{r_1}{r_2}\left(\frac{k(B_1)}{k(B_2)}\right)^{1/q}
 \leq C
\frac{r_1}{r_2}\left(\frac{|B_1|}{|B_2|}\right)^{\frac{1}{qt'}} 
 \leq C\left(\frac{r_1}{r_2}\right) ^{\frac{n}{qt'}+1} \\
 = C\left(\frac{r_1}{r_2}\right) ^{\frac{(s-\epsilon)Q}{p}} 
 \leq C\left(\frac{|B_1|}{|B_2|}\right)^{\frac{s-\epsilon}{p}} 
 \leq C\left(\frac{k^{-p/p'}(B_1)}{k^{-p/p'}(B_2)}\right)^{1/p},
\end{multline*}
 again with uniform constants by Remark~\ref{remark:constants}.
Since this is true for every such ball $B$,  $(k^{-p/p'},k)$ satisfies \eqref{cond2admissible}  of
Definition~\ref{defn:admissible}.

We now prove that \eqref{eqn:hb0} implies \eqref{eqn:hb1}.  By
Lemma~\ref{lemma:rh-ap-combo},
\[ k^\tau \in A_{\tau(p'-1)+1}(B,\rho,dx).  \]
Let $\sigma=\tau(p'-1)+1$.   Then a straightforward calculation shows that
\[ \sigma' = \frac{\tau(p'-1)+1}{\tau(p'-1)} =
\frac{p^2+np-p}{n-Q+pQ}, \]
and by Lemma~\ref{prop-Apfacts},
\[ k^{-p/p'} = k^{\tau(1-\sigma')} \in A_{\sigma'}(B,\rho,dx). \]
On the other hand if we replace $t$ by $\tau$ in~\eqref{eqn:hb1} and
take equality, then solving for $s$ we get (after another calculation)
\[ s = \frac{n}{Q}\left(1-\frac{1}{\tau}\right)+\frac{p}{Q} =
\frac{p^2+np-p}{n-Q+pQ} =\sigma'. \]
Note that $s>\frac{p}{Q}$ since $\tau>1.$ Therefore, $(k^{-p/p'},k)$ satisfies condition \eqref{cond2admissible}  of Definition~\ref{defn:admissible} and our proof is complete.
\end{proof}

\begin{proof}[Proof of Theorems~\ref{thm:main-CCS} and \ref{thm:main-carnot}]
  Our proof is a generalization of the proof of the main theorem
  in~\cite{cruz-diGironimo-sbordone-P}.  By Theorems
  \ref{thm:hormander:balance}, \ref{thm:carnot-balance}, and part
  \eqref{prop-Apfacts-a} of Lemma~\ref{prop-Apfacts} the pair of
  weights $(k^{-p/p'},k)$ is $p$-admissible in $\Omega$.  Let
  $B(x,r_0)\subseteq\Omega$ be a ball with sufficiently small radius
  so that Theorems \ref{thm:strong-mvt} and \ref{thm:harnack} hold;
  particular we may assume $r_0\leq 1$.  Set $B=B(x,r)$ where $r\leq
  r_0.$ Arguing as in the proof of Theorem \ref{thm:easy-thm} we
  obtain
$$\osc_u(x,r/2)\leq \frac{\exp(C_{r_0}(x)\mu_p(\frac12B))-1}{\exp(C_{r_0}(x)\mu_p(\frac12B))+1}\osc_u(x,r),$$
where the constant $C_{r_0}(x)$ depends on $x$ and $r_0$, but is finite for each $x$.
For any ball $B$ containing $x$,  H\"older's inequality yields
$$\mu_p(B)=\Big(\frac{\dashint_B k\,dy}{\dashint_B
  k^{-p/p'}\,dy}\Big)^{1/p}\leq \dashint_B k\,dy\leq M^*k(x);$$
hence,
\begin{equation}\label{eqn:osc} \osc_u(x,r/2)\leq
  \frac{\exp(C_{r_0}(x)M^*k(x))-1}{\exp(C_{r_0}(x)M^*k(x))+1}\osc_u(x,r)=\gamma(x)\osc_u(x,r).\end{equation}
We will now prove that the precise representative of $u$ is continuous on the set $S(k)$.
Notice that 
$$S(k)=\{x\in \Omega:M^*k(x)<\infty\}=\{x\in \Omega:\gamma(x)<1\}$$ and fix $x\in S(k)$.  Therefore, since $\gamma(x)<1$, we may apply Lemma 8.23 in \cite{MR1814364} to inequality \eqref{eqn:osc} to obtain
$$\osc_u(x,r)\leq c(x)\Big(\frac{r}{r_0}\Big)^{\alpha(x)}\,\osc_u(x,r_0), \qquad 0<r\leq r_0,$$
where $c(x)$ and $\alpha(x)$ are both finite constants depending on $x$.  Set 
$$u_{B(x,r)}=\dashint_{B(x,r)} u(y)\,dy$$ 
and notice that  $\lim_{r\ra 0} u_{B(x,r)}$ exists.  Indeed if $0<s<t<r_0$, then
\begin{align*}
|u_{B(x,t)}-u_{B(x,s)}| &\leq \dashint_{B(x,t)}\dashint_{B(x,s)}|u(z)-u(y)|\,dydz \\
&\leq \osc_u(x,t)\\
&\leq c(x)\Big(\frac{t}{r_0}\Big)^{\alpha(x)} \osc_u(x,r_0)\ra 0 \qquad \text{as} \ t\ra 0.
\end{align*}
If we set $\tilde{u}(x)=\lim_{r\ra 0} u_{B(x,r)}$ then a similar
argument shows that $\tilde{u}$ is the precise representative of $u$.
To see that $\tilde{u}$ is continuous on $S(k)$, let $B$ be a ball in
$\Omega$ with sufficiently small radius.  Then for any
$y\in S(k)\cap B$ we have
$$\essinf_B u\leq \tilde{u}(y)=\lim_{r\ra 0}\dashint_{B(y,r)} u(z)\,dz \leq \esssup_B u,$$
which implies that the pointwise oscillation of $\tilde{u}$ on $S(k)\cap B$ is dominated by the essential oscillation of $u$ on $B$.  The continuity of $\tilde{u}$ on $S(k)$ follows at once.
\end{proof}

\section{Examples}\label{sharp}



We now show that Theorem~\ref{thm:main-CCS} is sharp in the sense that there  exist $\Omega,$ $X,$  $A,$ $k$ and $p$ satisfying its hypothesis and a solution $u$ to the corresponding equation that fails to be continuous in a nonempty set of measure zero.  We will work in the euclidean setting, that is, $X=\nabla$ and address the case $p = 2$.  In this setting the ellipticity condition \eqref{eqn:CCS1} is given by
\begin{equation}\label{eqn:ellip:p=2} k(x)^{-1} |\xi|^2 \leq \langle A(x)\xi,\xi\rangle \leq k(x)|\xi|^2. \end{equation}
and \eqref{eqn:easy2} becomes
\begin{equation}\label{eqn:deg:laplace}\Div \big(A\grad u \big) = 0. \end{equation}
The examples presented below are based on those given in Zhong~\cite{MR2435212}.  We give an example when $n\geq 3$. An example when $n=2$ can also be constructed in a similar fashion using another example from \cite{MR2435212}.

\begin{example} \label{example:n>2} Let $\varepsilon>0$ and $\Omega=\{x=(x_1,\cdots,x_n)\in \R^n: |x|<\frac{1}{e}\}=B(0,1/e)$.  Define $\Omega_1=\{x\in\Omega : 2\,|x_n|>|x|\},$  
$\Omega_2=\{x\in\Omega : 2\,|x_n|<|x|\},$ and $A(x)=\tau(x)\,Id,$ where $Id$ is the $n\times n$ identity matrix and 
\[
\tau(x)=  \left\{ \begin{array}{ll} 
      1 & \text{ if }x\in\Omega_1,\\
      |\log|x||^{-(1+\varepsilon)} & \text{ if } x\in\Omega_2.
   \end{array} \right.
\] 
Zhong~\cite[Theorem 1.2]{MR2435212} proved that the equation \eqref{eqn:deg:laplace}
 has a weak solution $u\in S^2(\Omega)$ that is discontinuous at the origin.
In order to prove the sharpness of Theorem~\ref{thm:main-CCS} it is enough to check that there exists $\varepsilon>0$ such  that the corresponding matrix $A$ satisfies the degenerate ellipticity condition \eqref{eqn:ellip:p=2} for some  $k\in A_2(\Omega, \rho,dx)\cap RH_n(\Omega,\rho,dx),$ where $\rho$ denotes Euclidean distance.
We have that 
\[
\langle A(x)\xi,\xi\rangle= \left\{ \begin{array}{ll} 
      |\xi|^2 & \text{ if }x\in\Omega_1,\\
      |\log|x||^{-(1+\varepsilon)} \,|\xi|^2& \text{ if } x\in\Omega_2,
   \end{array} \right.
\]
which implies 
 \[
\min\left(1,|\log|x||^{-(1+\varepsilon)}\right)\,|\xi|^2\le \langle A(x)\xi,\xi\rangle\le \max\left(1,|\log|x||^{-(1+\varepsilon)}\right)\,|\xi|^2
\]
for $ x\in \Omega$ and $\xi\in\R^n.$ Since $\min\left(1,|\log|x||^{-(1+\varepsilon)}\right)=|\log|x||^{-(1+\varepsilon)}$ and \\ $ \max\left(1,|\log|x||^{-(1+\varepsilon)}\right)=1\le |\log|x||^{1+\epsilon},$ $x\in \Omega,$ we obtain
\[
k(x)^{-1}\,|\xi|^2\le \langle A(x)\xi,\xi\rangle\le k(x)\,|\xi|^2,\quad x\in \Omega,\,\xi\in\R^n,
\]
with $k(x)=|\log|x||^{1+\varepsilon}.$ 

We now show that, for any $\varepsilon>0,$  $k\in A_2(\Omega,\rho,dx)\cap RH_n(\Omega,\rho,dx).$  Consider the function $h(x)= \big|\log|x|\big|$ if $|x|< {1}/{e}$ and  $1$ otherwise. Then $h^s\in A_1(\R^n,\rho,dx)$ for all $s\ge 1$: that is,
\begin{equation}\label{claimforh}
\frac{1}{|B|}\int_B h(x)^s\,dx\sim \inf_{x\in B} h(x)^s.
\end{equation}
The proof of \eqref{claimforh} is elementary by considering balls
$B(x_0,r)$ with $|x_0|<cr$ and $|x_0|\geq cr$.  Note that this says
that $h\in RH_s(\R^n,\rho,dx)$ for all $s\ge 1.$   Given
\eqref{claimforh},  we use that $|B\cap\Omega|\sim |B|$ for all
Euclidean balls $B$ with center in $\Omega$ and radius $r(B)\leq C
\text{diam}(\Omega)$ to get that the restriction of $h^s$ to $\Omega$
is in $A_1(\Omega,\rho,dx)$ for all $s\ge 1.$ In particular, taking
$s=1+\epsilon,$ gives that $k\in A_1(\Omega,\rho,dx)$ and taking
$s=(1+\epsilon)n$ yields that $k^n\in A_1(\Omega,\rho,dx),$ from which
we conclude that $k\in A_2(\Omega,\rho,dx)\cap RH_n(\Omega,\rho,dx). $

\end{example}

\section{Application to Mappings of finite distortion}
\label{section:finite-distortion}

Recall that $f:\Omega\ra \R^n$ with $f\in W_{loc}^{1,1}(\Omega,\R^n)$ is a mapping of finite distortion if there is a function $K\geq 1$ finite a.e. such that
\begin{equation*}\label{distortion} |Df(x)|^n\leq K(x)J_f(x)\end{equation*}
where $|Df(x)|=\sup_{|h|=1} |Df(x)h|$ is the operator norm and
$J_f(x)=\det Df(x)$ is locally integrable. Throughout this section we will assume that
$n\geq 3$, as much of the planar theory is known.

The study of continuity of such functions is an active line of
research initiated by Vodop'janov and Gol'd{\v{s}}te{\u\i}n \cite{MR0414869}.  They showed that
if $f\in W^{1,n}_{loc}(\Omega,\R^n)$, then $f$ has a representative
that is continuous on $\Omega$.  More recent results include those of Manfredi
\cite{MR1294334}, Heinonen and Koskela \cite{MR1241287}, Iwaniec,
Koskela and Onninen \cite{MR1833892} 
and Kauhanen, Koskela, Mal\'y and
Zhong~\cite{MR2053566}.  Without assuming higher
integrability on $|Df|$, many of the continuity results assume some
sort of exponential integrability condition on $K$, namely $\exp(K)\in L^p_{loc}(\Omega)$
for some $0<p<\infty$.  Broadly speaking, to obtain continuity one
should not stray too far from the assumptions $f\in
W_{loc}^{1,n}(\Omega,\R^n)$ or $K\in$ Exp $L$, (see \cite{MR2053566}).

Our results confirm this perception.  We are able to work below the
natural thresholds of $W_{loc}^{1,n}(\Omega,\R^n)$ and exponential
integrability of the distortion function.  But in doing so we do not
obtain functions which are continuous everywhere, but rather
continuous except on a set of measure zero.  Our results are in the spirit of 
 those in \cite{MR1294334} where it was shown that given a mapping
of finite distortion $f\in W_{loc}^{1,p}(\Omega,\R^n)$ for some $p$,
$n-1<p<n$, then $f$ is continuous except on a set of $p$-capacity
zero.  On the other hand, our natural assumption is that $f\in
W^{1,n-1}_{loc}(\Omega,\R^n)$ and the distortion function belongs to
an appropriate weight class.

To state our results we briefly sketch some basic facts about mappings
of finite distortion.  A more detailed exposition can be found in
\cite{MR1859913}.  Given  a mapping of finite distortion $f$, we will be
concerned with two distortion functions: the inner, denoted $K_I$, and the outer, denoted $K_O$.  They are defined as follows 
$$K_I(x)=\frac{|\text{adj}\, Df(x)|^n}{J_f(x)^{n-1}}  \quad \text{and} \quad K_O(x)=\frac{|Df(x)|^n}{J_f(x)}. $$
Here adj $A$ denotes the adjoint of the matrix $A$ defined by the equation
$A\, \text{adj}\, A=(\det A)\, \text{Id}$.  The distortion tensor 
$$G(x)=\displaystyle\frac{Df(x)^TDf(x)}{J_f(x)^{2/n}} $$
satisfies the ellipticity condition
\begin{equation} \label{elliptic:out:in}{K_O(x)^{-2/n}}{|\xi|^2}\leq \langle G(x)^{-1}\xi,\xi\rangle \leq K_I(x)^{2/n}|\xi|^2.\end{equation}
Since in general $K_I(x)^{1/(n-1)}\leq K_O(x)\leq K_I(x)^{n-1},$ 
we have the following ellipticity condition on the inner distortion function:
\begin{equation} \label{ellipinner}{K_I(x)^{-2/n'}}{|\xi|^2}\leq
  \langle G(x)^{-1}\xi,\xi\rangle \leq K_I(x)^{2/n}|\xi|^2. 
\end{equation}
But this is precisely  the ellipticity condition \eqref{eqn:CCS1} when $p=n$ and $k=K_I$.

\begin{theorem} \label{finitedistortion:thm} Suppose $f:\Omega \ra
  \R^n$ is a mapping of finite distortion whose coordinate functions
  belong to $S^n(\Omega)\cap W^{1,n-1}_{loc}(\Omega)$ and whose inner
  distortion function satisfies $K_I\in A_{n'}(\Omega)\cap
  RH_n(\Omega)$.  (Equivalently, $K_I^n \in A_{n'+1}(\Omega)$.)  
 Then $f$ has a representative which is continuous in the set
$$\{x\in \Omega: MK_I(x)<\infty\}.$$
\end{theorem}
\begin{remark} The assumption $f=(f^1,\ldots,f^n)\in W^{1,n-1}_{loc}(\Omega,\R^n)$ is the very minimal assumption to give meaning to the equation
$$\Div(\langle G(x)^{-1}\nabla f^i,\nabla f^i\rangle^{\frac{n-2}{2}} G(x)^{-1}\nabla f^i)=0 \qquad i=1,\ldots,n$$
(see \cite[p. 379]{MR1859913}).  Since we are also assuming that $f^i\in S^n(\Omega)$ the assumption $f\in W^{1,n-1}_{loc}(\Omega,\R^n)$ is superfluous if $K_O\in L^{n-1}_{loc}(\Omega)$.   In this case $S^n(\Omega)$ embeds into $W^{1,n-1}_{loc}(\Omega)    $.  Indeed, if $g\in S^n(\Omega)$ and $K_O\in L^{n-1}_{loc}(\Omega)$ then the ellipticity condition \eqref{elliptic:out:in} implies
\begin{align*}
\Big(\int_B |\nabla g|^{n-1}\,dx\Big)^{1/(n-1)}&\leq\Big( \int_B|\nabla g|^nK_O^{-1}\,dx\Big)^{1/n}\Big(\int_B K_O^{n-1}\,dx\Big)^{1/n(n-1)}\leq C\|\nabla g\|_{G^{-1}}
\end{align*}
(see \cite{MR1241287} for similar embeddings).
\end{remark}
\begin{proof} We will show that each of the coordinate functions $f=(f^1,\ldots,f^n)$ is a weak solution to the equation $\Lap_nu=-\Div(\langle A\nabla u, \nabla u\rangle^{\frac{n-2}{2}}A\nabla u)=0,$ where $A(x)=G(x)^{-1}$.  Fix a coordinate function $f^i$, and let
$u_k$ be a sequence of functions in $\text{Lip}(\overline{\Omega})$
converging to $f^i$ in the norm of $S^n(\Omega)$, and for the moment
let $\vp\in C_0^\infty$.  By Lemma \ref{lemma:convergence} we have
$$a_0^n(f^i,\vp)=\lim_{k\ra \infty} a^n_0(u_k,\vp)=
\int_\Omega \langle A\nabla f^i,\nabla f^i\rangle^{\frac{n-2}{2}}\langle A\nabla f^i,\nabla \vp\rangle\,dx.$$
Since $\nabla f^i$ is the distributional gradient,  a calculation
shows (see \cite[p. 379]{MR1859913}) that 
\[\langle A(x)\nabla f^i(x),\nabla f^i(x)\rangle^{\frac{n-2}{2}}
  A(x)\nabla f^i(x) =[\text{adj} \,Df(x)]_i
\]
where $[\text{adj}\, Df(x)]_i$ denotes the $i$-th column of the
adjoint matrix of $Df(x)$.  Since $f\in W^{1,n-1}_{loc}(\Omega,\R^n)$
it is well known that (see \cite[p. 256]{MR1241287} or
\cite[p. 66]{MR1859913})
$$\int_\Omega \langle[\text{adj}\, Df(x)]_i, \nabla \vp\rangle,  \,
dx=0 \qquad i=1,\ldots,n,$$
for all $\vp\in C_0^\infty(\Omega)$.  Since $A$ satisfies
\eqref{ellipinner}, by Theorem \ref{thm:main-carnot} we get the
desired result.
\end{proof}
Finally we end with an example of a mapping of finite distortion that is discontinuous on a set of measure zero whose inner distortion function belongs to $A_{n'}\cap RH_n$.
\begin{example} Let $B=B(0,1)=\{x\in \R^n: |x|<1\}$ and consider 
$$f(x)=\frac{x}{|x|} \exp(|x|^\varepsilon)$$
where $\varepsilon>0$ is a small quantity to be chosen later.  Notice
that there is no way to define $f$ so
that it is continuous at the origin.  Calculations show (see
\cite[chapter 6]{MR1859913} for more details) that
$$|Df(x)|=|x|^{-1}{\exp(|x|^\varepsilon)} \quad \text{and} \quad J_f(x)= \varepsilon|x|^{\varepsilon-n}\exp(n|x|^\varepsilon).$$
We have that $f\in W^{1,p}(B,\R^n)$ for all $1\leq p<n$ but $f\not\in
W^{1,n}(B,\R^n)$.  Moreover,
$$K_O(x)={\varepsilon^{-1}} {|x|^{-\varepsilon}}\quad \text{and} \quad K_I(x)=(\varepsilon^{-1} |x|^{-\varepsilon})^{n-1}.$$
If $\ep>0$ then neither $K_I$ nor $K_O$ belong to $Exp\, L$.  However, if $\varepsilon<\frac{1}{n-1}$ then $K_I\in A_{n'}\cap RH_n$ and $MK_I(x)\sim |x|^{-\varepsilon(n-1)};$
thus $f$ is discontinuous at $x=0$,  exactly where $MK_I=\infty$.

\end{example}

\bibliographystyle{plain}
\bibliography{continuity}

\end{document}